\documentclass[12pt]{amsart}
\usepackage{amssymb, verbatim}
\usepackage{epsfig,color}
\usepackage{enumerate}
\usepackage{color}
 
\date{\today}

\newtheorem{theorem}{Theorem}[section]

\newtheorem{lemma}[theorem]{Lemma}

\theoremstyle{definition}

\theoremstyle{remark}
\newtheorem{remark}{Remark}

\makeatletter
\@namedef{subjclassname@2020}{%
 \textup{2020} Mathematics Subject Classification}
\makeatother

\numberwithin{equation}{section}

\usepackage[T1]{fontenc}

\newcommand{\hyper}[5]{\,{}_{#1}F_{#2}\left(\!\!%
\begin{array}{cc}{\displaystyle{#3}}\\[-0.1ex]
{\displaystyle{#4}} \end{array}\Big| \,{\displaystyle{#5}}
\right)}

\begin{document}

\title[Fourier transform of some special functions]{Fourier transforms of some special functions in terms of orthogonal polynomials on the simplex and continuous Hahn polynomials}

\author[G\"{u}ldo\u{g}an-Lekesiz]{Esra G\"{u}ldo\u{g}an-Lekesiz}
\address[G\"{u}ldo\u{g}an-Lekesiz]{Atilim University, Department of Mathematics, Incek 06836, Ankara, Turkey}
\email[G\"{u}ldo\u{g}an-Lekesiz]{esra.guldogan@atilim.edu.tr}

\author[Akta\c{s}]{Rabia  Akta\c{s}}
\address[Akta\c{s}]{Ankara University, Faculty of Science, Department of Mathematics, 06100, Tando\u{g}an, Ankara, Turkey}
\email[Akta\c{s}]{raktas@science.ankara.edu.tr}

\author[Area]{I. Area}
\address[Area]{ Universidade de Vigo, Departamento de Matem\'atica Aplicada II,
 E.E. Aeron\'autica e do Espazo,
 Campus As Lagoas-Ourense,
 32004 Ourense, Spain.}
\email[Area]{area@uvigo.gal}

\subjclass[2020]{Primary 33C50 \ Secondary 33C45} 

\keywords{Jacobi polynomials; Multivariate orthogonal polynomials; Hahn polynomials; Fourier transform; Parseval's identity; hypergeometric function; recurrence relation}
\begin{abstract}
In this paper Fourier transform of multivariate orthogonal polynomials on the simplex is presented. A new family of multivariate orthogonal functions is obtained by using the Parseval's identity and several recurrence relations are derived.
\end{abstract}

\maketitle

\section{Introduction}

The integral transforms are important mathematical tools used in many fields such as vibration analysis, sound engineering and image processing. They have wide applications in physics, engineering and in other scientific and mathematical disciplines. The study of orthogonal polynomials and their transformations have been the subject of many papers during the last several years.
By the Fourier transform or other integral transforms, some univariate orthogonal polynomials systems which are mapped onto each other exist \cite{MR0065685}. As an example, it is known that Hermite functions which are multiplied Hermite polynomials $H_{n}\left(  x\right)$ by $\exp \left(  -x^{2}/2\right)$ are eigenfunctions of Fourier transform \cite{MR1307541,MR2336575, MR973421, MR838982}; similarly, Koelink \cite{MR1307541} has showed that Jacobi polynomials are mapped onto the continuous Hahn polynomials by using the Fourier transform. Also, in \cite{MR838982}, it is seen that classical Jacobi polynomials can be mapped onto Wilson polynomials by the Fourier-Jacobi transform. By inspired of Koelink's paper, in \cite{MR2888193} Masjed-Jamei et al. have introduced two new families of orthogonal functions by using Fourier transforms of the generalized ultraspherical polynomials and the generalized Hermite polynomials. In \cite{MR3170751}, the Fourier transform of Routh-Romanovski polynomials $J_{n}^{\left(u,v\right)  }\left(x;a,b,c,d\right)$ has been obtained. Furthermore, in \cite{MR2336575, MR3169170}, four new examples of finite orthogonal functions have been derived by use of the Fourier transforms of the finite classical orthogonal polynomials $M_{n}^{\left(p,q\right)  }\left(  x\right)$ and $N_{n}^{\left(  p\right)  }\left(x\right)$, and two symmetric sequences of finite orthogonal polynomials and the Parseval's identity (see \cite{MR2246501, MR1915513, MR2270049} for details of the finite orthogonal polynomials and symmetric sequences of finite orthogonal polynomials). Recently, G\"{u}ldo\u{g}an et al. \cite{GAA} have obtained new families of orthogonal functions by means of Fourier transforms of bivariate orthogonal polynomials introduced by Koornwinder \cite{MR0402146} and Fern\'andez et al. \cite{MR2923514}.\\
In the present paper, we study the Fourier transformations of the classical polynomials orthogonal by means of the extension of Jacobi weight function to several variables on the simplex $T^{r}$. When $r=1$, the simplex $T^{r}$ becomes the interval $[0,1]$ and the corresponding orthogonal polynomials are Jacobi polynomials $P_{n}^{\left(  \alpha,\beta \right)  }\left(  2x-1\right)$ on the interval $[0,1]$. In this case, Koelink \cite{MR1307541} has derived the Fourier transform of certain Jacobi polynomials $P_{n}^{\left(  \alpha,\beta \right)  }\left(  x\right)$ on the interval $[-1,1]$ in terms of continuous Hahn polynomials and has discussed some applications. The Fourier transforms of bivariate orthogonal polynomials have been studied by G\"{u}ldo\u{g}an et al. \cite{GAA}. By the motivation of these papers, we obtain the Fourier transforms of the orthogonal polynomials on $T^{r}$ and we write them in terms of continuous Hahn polynomials. We first state the results for $r=1$ and $r=2$ to illustrate the results and illuminate how the results on $T^{r}$ are obtained, then we give the results on $T^{r}$ by induction. By using the obtained Fourier transforms and Parseval's identity, we derive a new family of orthogonal functions and obtain several recurrence relations for these functions.

The main aim of this work is to derive the Fourier transforms of functions defined in terms of orthogonal polynomials on the simplex and to obtain a new family of multivariate orthogonal functions by the similar method applied in \cite{GAA} for bivariate Koornwinder polynomials. While doing these, firstly we define specific special functions so that they are determined with the motivation to use the orthogonality relation of orthogonal polynomials on the simplex in Parseval's identity created with the help of Fourier transform. In order to give the results on $T^{r}$, we shall proceed by induction and we first discuss the results for $r=1$ and $r=2$.  The manuscript is organized as follows: in section \ref{sec:bdn} we introduce some basic definitions and notations. Section \ref{sec:ft} includes the Fourier transforms of the multivariate orthogonal polynomials on the $r$-simplex. In section \ref{sec:ops},  by using the Parseval's identity, a new family of multivariate orthogonal functions is defined by induction. Finally, in section \ref{sec:fs} we derive several recurrence relations for the new family of multivariate orthogonal functions.

\section{Basic definitions and notations}\label{sec:bdn}

The univariate Jacobi polynomials $P_{n}^{\left(  \alpha,\beta \right)  }\left(  x\right) $ are defined by the explicit representation
\begin{equation}\label{A:jac}
P_{n}^{\left(  \alpha,\beta \right)  }\left(  x\right)  =2^{-n}{\displaystyle \sum \limits_{k=0}^{n}} \binom{n+\alpha}{k} \binom{n+\beta}{n-k} \left(  x+1\right)  ^{k}\left(  x-1\right)  ^{n-k}.
\end{equation}
They are orthogonal on the interval $\left[  -1,1\right]$ with respect to the weight function $w\left(  x\right)  =\left(  1-x\right)^{\alpha}\left(  1+x\right)  ^{\beta}$ (cf. \cite{MR0393590,MR0372517}),
\begin{multline*}
{\displaystyle \int \limits_{-1}^{1}}\left(  1-x\right)  ^{\alpha}\left(  1+x\right)  ^{\beta}P_{n}^{\left(\alpha,\beta \right)  }\left(  x\right)  P_{m}^{\left(  \alpha,\beta \right)
}\left(  x\right)  dx \\
=\frac{2^{\alpha+\beta+1}\Gamma \left(  \alpha+n+1\right)\Gamma \left(  \beta+n+1\right)  }{n!\left(  \alpha+\beta+2n+1\right)\Gamma \left(  \alpha+\beta+n+1\right)  }\delta_{m,n}, \qquad \alpha,\beta>-1,
\end{multline*}
where $m,n\in \mathbb{N}_{0}:=\mathbb{N\cup}\left \{  0\right \}$,  $\delta_{m,n}$ is the Kronecker delta, and $\Gamma(\lambda)$ is the Gamma function (cf. \cite{MR757537}) defined by
\begin{equation}
\Gamma \left(  \lambda \right)  ={\displaystyle \int \limits_{0}^{\infty}}x^{\lambda-1}e^{-x}dx, \quad \Re\left(  \lambda \right)  >0.
\end{equation}

The $(p,q)$ generalized hypergeometric function is defined by (cf. \cite{MR757537})
\begin{equation}\label{hyper}
\hyper{p}{q}{a_{1},a_{2},\dots,a_{p}}{b_{1},b_{2},\dots,b_{q}}{x}={\displaystyle \sum \limits_{k=0}^{\infty}} \frac{\left(  a_{1}\right)  _{k}\left(  a_{2}\right)  _{k} \cdots \left(
a_{p}\right)  _{k}}{\left(  b_{1}\right)  _{k}\left(  b_{2}\right)_{k} \cdots \left(  b_{q}\right)  _{k}}\frac{x^{k}}{k!},
\end{equation}
where $\left(  a\right)_{k}={\displaystyle \prod \limits_{r=0}^{k-1}} \left(  a+r\right)$, $\left(  a\right)  _{0}=1$, $k\in \mathbb{N}_{0}$, is the Pochhammer symbol, and the beta function is given by (cf. \cite{MR757537})
\begin{equation}
B\left(  a,b\right)  ={\displaystyle \int \limits_{0}^{1}}x^{a-1}\left(  1-x\right)  ^{b-1}dx=\frac{\Gamma \left(  a\right)\Gamma \left(  b\right)  }{\Gamma \left(  a+b\right)  }, \qquad \Re\left(  a\right)  ,\Re\left(  b\right)  >0.
\end{equation}

The continuous Hahn polynomials can be introduced in terms of hypergeometric series as \cite{MR812420}
\begin{equation}\label{hahn}
p_{n}\left( x;a,b,c,d\right) =i^{n}\frac{\left( a+c\right) _{n}\left(a+d\right) _{n}}{n!} \hyper{3}{2}{-n,\ n+a+b+c+d-1,\ a+ix}{a+c,\ a+d}{1}.
\end{equation}
These polynomials can be presented as a limiting case of the Wilson polynomials \cite{MR812420}. Koelink \cite{MR1307541} gave a nice survey of the history this family of orthogonal polynomials. Recently \cite{Alhaidari} some applications of continuous Hahn polynomials to quantum mechanics have been presented.

Let $\left \vert \boldsymbol{x}\right \vert :=x_{1}+\cdots +x_{r}$ for $\boldsymbol{x}=\left(  x_{1},\dots,x_{r}\right)\in \mathbb{R}^{r}$. Let $V_{n}^{r}$ be the linear space of orthogonal polynomials respect to the weight function $W_{\mathbf{\alpha}}\left(  \boldsymbol{x}\right) =x_{1}^{\alpha_{1}} \cdots x_{r}^{\alpha_{r}}\left(  1-\left \vert \boldsymbol{x}\right \vert \right)  ^{\alpha_{r+1}}$ , $\alpha=(\alpha_{1},\dots ,\alpha_{r+1})$, $\alpha_{i}>-1$, on the simplex $T^{r}=\left \{  \boldsymbol{x}\in
\mathbb{R} ^{r}:x_{1}\geq0,\dots,x_{r}\geq0,1-\left \vert \boldsymbol{x}\right \vert \geq0\right \}  $. The elements of the linear space $V_{n}^{r}$ verify the following partial differential equation (cf. \cite{MR3289583})
\begin{multline*}
 \sum \limits_{i=1}^{r}x_{i}\left(  1-x_{i}\right)  \frac{\partial^{2} P}{\partial x_{i}^{2}}-2\sum \limits_{1\leq i<j\leq r}x_{i}x_{j}\frac{\partial^{2}P}{\partial x_{i}\partial x_{j}}
 +\sum \limits_{i=1}^{r}\left(  \left(  \alpha_{i}+1\right)  -\left(\left \vert \mathbf{\alpha}\right \vert +r+1\right)  x_{i}\right)\frac{\partial P}{\partial x_{i}}\\
 =-n\left \{  n+\left \vert \mathbf{\alpha}\right \vert +r+1\right \}  P,
\end{multline*}
which has $\dbinom{n+r-1}{n}$ linear independent polynomial solutions of total degree $n$. In here, $\left \vert \mathbf{\alpha}\right \vert =\alpha_{1}+\cdots+\alpha_{r+1}$. The space $V_{n}^{r}$ has several different bases. One of these bases can be introduced as follows. Let $\mathbf{x}_{j}$ and $\mathbf{n}^{j}$ be defined as
\[
\boldsymbol{x}_{0}=0,\quad  \boldsymbol{x}_{j}=\left(  x_{1},\dots,x_{j}\right), \quad \left \vert \boldsymbol{x}_{j}\right \vert =x_{1}+\cdots+x_{j}, \quad 1\leq j\leq r,
\]
\[
\mathbf{n}^{j}=\left(  n_{j},\dots,n_{r}\right), \quad \left \vert\mathbf{n}^{j}\right \vert =n_{j}+\cdots+n_{r}, \quad 1\leq j\leq r
\]
and $\mathbf{n}^{r+1} :=0$.

Also, let $\mathbf{n}=\left(  n_{1},\dots,n_{r}\right) $, $\left \vert\mathbf{n}\right \vert =n_{1}+\cdots+n_{r}=n$, $\mathbf{\alpha}=\left(  \alpha_{1},\dots,\alpha_{r+1}\right) $, and $\mathbf{\alpha}^{j}=\left(  \alpha_{j},\dots,\alpha_{r+1}\right)$, $\left \vert\mathbf{\alpha}^{j}\right \vert =\alpha_{j}+\cdots+\alpha_{r+1}$,  $1\leq j\leq r+1$. In this case, an orthogonal base of the space $V_{n}^{r}$ is of form (cf. \cite{MR3289583})
\begin{equation}\label{A:100}
P_{\mathbf{n}}^{\left(  \mathbf{\alpha}\right)  }\left(  \boldsymbol{x}\right)  =\prod \limits_{j=1}^{r}\left(1-\left \vert \boldsymbol{x}_{j-1}\right \vert \right)
^{n_{j}}P_{n_{j}}^{\left(  a_{j},b_{j}\right)  }\left(  \frac{2x_{j}}{1-\left \vert \boldsymbol{x}_{j-1}\right \vert }-1\right),
\end{equation}
where $P_{n_{j}}^{\left(  a_{j},b_{j}\right)  }\left(  x_{j}\right)  $ indicates the classical univariate Jacobi polynomial defined in \eqref{A:jac}, $a_{j}=2\left \vert \mathbf{n}^{j+1}\right \vert +\left \vert \mathbf{\alpha}^{j+1}\right \vert +r-j$ and $b_{j}=\alpha_{j}$. It follows
\begin{equation}\label{A:3}
{\displaystyle \int \limits_{T^{r}}}W_{\alpha}\left(  \boldsymbol{x}\right)  P_{\mathbf{n}}^{\left(\mathbf{\alpha}\right)  }\left(  \boldsymbol{x}\right)  P_{\mathbf{m}}^{\left(  \mathbf{\alpha}\right)  }\left(  \boldsymbol{x}\right)\mathbf{dx=}h_{\mathbf{n}}^{\left(  \mathbf{\alpha}\right)  }\delta_{\mathbf{n,m}},
\end{equation}
where $\mathbf{dx=}dx_{1}\cdots dx_{r}$, and
\begin{multline}\label{B}
h_{\mathbf{n}}^{\left(  \mathbf{\alpha}\right)  }=
\prod \limits_{j=1}^{r}\frac{\Gamma \left(  2\left \vert \mathbf{n}^{j+1}\right \vert +\left \vert \mathbf{\alpha}^{j+1}\right \vert +r-j+1+n_{j}\right)  }{n_{j}!\left(  2\left \vert \mathbf{n}^{j+1}\right \vert +\left \vert \mathbf{\alpha}^{j+1}\right \vert +r-j+1+\alpha_{j}+2n_{j}\right)  } \\
\times \frac{ \Gamma \left(1+\alpha_{j}+n_{j}\right) }{\Gamma \left(  2\left \vert \mathbf{n}^{j+1}\right \vert+\left \vert \mathbf{\alpha}^{j+1}\right \vert +r-j+1+\alpha_{j}+n_{j}\right)} .
\end{multline}

\section{Fourier Transforms of Specific Functions}\label{sec:ft}

The Fourier transform for a function $g(x)$ is defined as (cf. \cite{MR0942661})
\begin{equation}\label{16}
\mathcal{F}
\left(  g\left(  x\right)  \right)  =\int \limits_{-\infty}^{\infty}e^{-i\xi x}g\left(  x\right)  dx
\end{equation}
and the Fourier tranform for a function $g(x_{1},\dots,x_{r})$ in $r$ variables is defined as (cf. \cite{MR1867914})
\begin{equation}\label{17}
\mathcal{F}
\left(  g\left(  x_{1},\dots,x_{r}\right)  \right)  =\int \limits_{-\infty}^{\infty}\dots\int \limits_{-\infty}^{\infty}e^{-i\left(  \xi_{1}x_{1}+\cdots+\xi_{r}x_{r}\right)  }g\left(  x_{1},\dots,x_{r}\right)  dx_{1}
\dots dx_{r}.
\end{equation}

\subsection{The Fourier transform of Jacobi polynomials on 1-simplex}
When $r=1$, the simplex $T^{r}$ becomes the interval $[0,1]$ and the corresponding orthogonal polynomials in \eqref{A:100} are the Jacobi polynomials $P_{n}^{\left( \alpha \right)  }\left( x\right):= P_{n}^{\left( \alpha_{2}, \alpha_{1} \right)  }\left(  2x-1\right)$.\\
Define
\begin{equation}\label{55}
g_{1}\left(  x;n,a_{1},a_{2},\alpha_{1},\alpha_{2}\right)
=\left(  1+\tanh x\right)  ^{a_{1}}\left(  1-\tanh x\right)  ^{a_{2}} P_{n}^{\left(\alpha_{1},\alpha_{2}\right)  }\left( \Upsilon_{1}\right)  ,
\end{equation}
where $a_{1},a_{2},\alpha_{1},$ and $\alpha_{2}$ are real parameters and
\begin{equation*}
\Upsilon_{1}(x)=\Upsilon_{1}= \frac{1+\tanh x}{2}.
\end{equation*}
 The Fourier transform of this function is
\begin{multline} \label{99}
\mathcal{F}\left( g_{1}\left( x;n,a_{1},a_{2},\alpha _{1},\alpha _{2}\right)
\right)  
=\int \limits_{-\infty }^{\infty }e^{-i\xi x}\left( 1+\tanh x\right)
^{a_{1}}\left( 1-\tanh x\right) ^{a_{2}}P_{n}^{\left( \alpha _{2},\alpha
_{1}\right) }\left( \tanh x\right) dx  \\
= \frac{2^{a_{1}+a_{2}-1}\left( \alpha _{2}+1\right) _{n}}{n!}\Lambda_{1}^{1}(\mathbf{a},\mathbf{\alpha},n;\xi),
\end{multline}
where
\begin{equation*}
\Lambda_{1}^{1}(\mathbf{a},\mathbf{\alpha},n;\xi)= \hyper{3}{2}{-n ,\ a_{2}+\frac{i\xi}{2},\
n+\alpha_{1}+\alpha_{2}+1}{\alpha_{2}+1,\ a_{1}+a_{2}}{1}  B\left( a_{2}+%
\frac{i\xi }{2},\ a_{1}-\frac{i\xi }{2}\right), \\
\end{equation*}
which was proved by Koelink \cite{MR1307541} and it was also rewritten in terms of the continuous Hahn polynomials \eqref{hahn}
as
\begin{multline*}
\mathcal{F}\left( g_{1}\left( x;n,a_{1},a_{2},\alpha _{1},\alpha _{2}\right)
\right)=\frac{2^{a_{1}+a_{2}-1}}{i^{n}\left( a_{1}+a_{2}\right) _{n}}B\left( a_{2}+%
\frac{i\xi }{2},\ a_{1}-\frac{i\xi }{2}\right)  \\
\times p_{n}\left( \frac{\xi }{2};\ a_{2},\  \alpha _{1}-a_{1}+1,\  \alpha
_{2}-a_{2}+1,\ a_{1}\right) ,
\end{multline*}
where $\mathbf{a}=\left( a_{1},a_{2}\right)$, $\mathbf{\alpha}=\left( \alpha_{1},\alpha_{2}\right)$ and $\ _{3}F_{2}$ is a special case of the hypergeometric function \eqref{hyper}.

\subsection{The Fourier transform of two-dimensional polynomials on 2-simplex}
Let consider the following specific function in terms of the orthogonal polynomials on 2-simplex
\begin{multline}\label{5}
g_{2}\left(  x_{1},x_{2};n_{1},n_{2},a_{1},a_{2},a_{3},\alpha_{1},\alpha_{2},\alpha_{3}\right)
=\left(  1+\tanh x_{1}\right)  ^{a_{1}}\left(  1-\tanh x_{1}\right)  ^{a_{2}+a_{3}} \\
\times \left(  1+\tanh x_{2}\right)  ^{a_{2}}   \left(  1-\tanh x_{2}\right)  ^{a_{3}}
 P_{n_{1},n_{2}}^{\left(\alpha_{1},\alpha_{2},\alpha_{3}\right)  }\left( \Upsilon_{1},\Upsilon_{2} \right)  ,
\end{multline}
which is selected in this form to use the orthogonality relation for orthogonal polynomials on 2-simplex in Parseval's identity by the motivation of the Fourier transform of Jacobi polynomials on 1-simplex. Here $a_{1},a_{2},a_{3},\alpha_{1},\alpha_{2},\alpha_{3}$ are real parameters, and
\begin{equation*}
\Upsilon_{1}(x_{1})=\Upsilon_{1}= \frac{1+\tanh x_{1}}{2},\qquad \Upsilon_{2}(x_{1},x_{2})=\Upsilon_{2}=\frac{\left(  1-\tanh x_{1}\right)  \left(  1+\tanh x_{2}\right)  }{4}.
\end{equation*}
It is easily seen that the function $g_{2}$ can also be expressed in terms of $g_{1}$ as follows:
\begin{multline}\label{555}
g_{2}\left(  x_{1},x_{2};n_{1},n_{2},a_{1},a_{2},a_{3},\alpha_{1},\alpha_{2},\alpha_{3}\right)  \\
=2^{-n_{2}}\left(  1+\tanh x_{2}\right)  ^{a_{2}}\left(  1-\tanh x_{2}\right)  ^{a_{3}}
 P_{n_{2}}^{\left(\alpha_{3},\alpha_{2}\right)  }\left( \tanh x_{2} \right)\\
 \times g_{1}\left(  x_{1};n_{1},a_{1},a_{2}+a_{3}+n_{2},\alpha_{1},\alpha_{2}+\alpha_{3}+2n_{2}+1\right).
\end{multline}

 The corresponding Fourier transform is calculated, under the substitution $%
\tanh x_{2}=2v-1$, as follows from
\begin{multline*}
\mathcal{F}\left( g_{2}\left(
x_{1},x_{2};n_{1},n_{2},a_{1},a_{2},a_{3},\alpha _{1},\alpha _{2},\alpha
_{3}\right) \right)  \\
=\int \limits_{-\infty }^{\infty }\int \limits_{-\infty }^{\infty }e^{-i\left(
\xi _{1}x_{1}+\xi _{2}x_{2}\right) }g_{2}\left(
x_{1},x_{2};n_{1},n_{2},a_{1},a_{2},a_{3},\alpha _{1},\alpha _{2},\alpha
_{3}\right) dx_{1}dx_{2} \\
=2^{-n_{2}}\int \limits_{-\infty }^{\infty }e^{-i\xi _{1}x_{1}}g_{1}\left(
x_{1};n_{1},a_{1},a_{2}+a_{3}+n_{2},\alpha _{1},\alpha _{2}+\alpha
_{3}+2n_{2}+1\right) dx_{1} \\
\times \int \limits_{-\infty }^{\infty }e^{-i\xi _{2}x_{2}}\left( 1+\tanh
x_{2}\right) ^{a_{2}}\left( 1-\tanh x_{2}\right) ^{a_{3}}P_{n_{2}}^{\left(
\alpha _{3},\alpha _{2}\right) }\left( \tanh x_{2}\right) dx_{2}
\end{multline*}
\begin{multline*}
=\mathcal{F}\left( g_{1}\left( x_{1};n_{1},a_{1},a_{2}+a_{3}+n_{2},\alpha
_{1},\alpha _{2}+\alpha _{3}+2n_{2}+1\right) \right)  \\
\times 2^{a_{2}+a_{3}-n_{2}-{1}}\int \limits_{0}^{1}v^{a_{2}-\frac{i\xi _{2}}{%
2}-1}\left( 1-v\right) ^{a_{3}+\frac{i\xi _{2}}{2}-1}P_{n_{2}}^{\left(
\alpha _{3},\alpha _{2}\right) }\left( 2v-1\right) dv \\
=\mathcal{F}\left( g_{1}\left( x_{1};n_{1},a_{1},a_{2}+a_{3}+n_{2},\alpha
_{1},\alpha _{2}+\alpha _{3}+2n_{2}+1\right) \right)  \\
\times 2^{a_{2}+a_{3}-n_{2}-{1}}\int \limits_{0}^{1}v^{a_{2}-\frac{i\xi _{2}}{%
2}-1}\left( 1-v\right) ^{a_{3}+\frac{i\xi _{2}}{2}-1}\frac{\left( \alpha
_{3}+1\right) _{n_{2}}}{n_{2}!} \\
\times \hyper{2}{1}{-n_{2}, n_{2}+\alpha_{2}+\alpha_{3}+1}{ \
\alpha_{3}+1 }{1-v}dv
\end{multline*}
\begin{multline}
=\mathcal{F}\left( g_{1}\left( x_{1};n_{1},a_{1},a_{2}+a_{3}+n_{2},\alpha
_{1},\alpha _{2}+\alpha _{3}+2n_{2}+1\right) \right) \frac{\left( \alpha _{3}+1\right) _{n_{2}}}{n_{2}!} 2^{a_{2}+a_{3}-n_{2}-{1}} \\
\times \sum \limits_{l_{2}=0}^{n_{2}}\frac{\left( -n_{2}\right)
_{l_{2}}\left( n_{2}+\alpha _{2}+\alpha _{3}+1\right) _{l_{2}}}{\left(
\alpha _{3}+1\right) _{l_{2}}l_{2}!}
  \int \limits_{0}^{1}v^{a_{2}-\frac{i\xi _{2}}{2}-1}\left( 1-v\right)
^{a_{3}+\frac{i\xi _{2}}{2}-1+l_{2}}dv
\end{multline}
that is
\begin{multline}
\mathcal{F}\left(  g_{2}\left(  x_{1},x_{2};n_{1},n_{2},a_{1},a_{2},a_{3},\alpha_{1},\alpha_{2},\alpha_{3}\right)  \right) \label{1}\\
=\frac {2^{a_{1}+2\left(  a_{2}+a_{3}\right) - {2}} \left(  2n_{2}+\alpha_{2}+\alpha_{3}+2\right)  _{n_{1}}\left(  \alpha_{3}+1\right) _{n_{2}}}{n_{1}!n_{2}!}\Lambda_{1}^{2}(\mathbf{a},\mathbf{\alpha},\mathbf{n};\xi_{1}) \Lambda_{2}^{2}(\mathbf{a},\mathbf{\alpha},\mathbf{n};\xi_{2}),
\end{multline}
where
\begin{multline*}
\Lambda_{1}^{2}(\mathbf{a},\mathbf{\alpha},\mathbf{n};\xi_{1})= B\left(  a_{1}-\frac{i\xi_{1}}{2},n_{2}+a_{2}+a_{3}+\frac{i\xi_{1}}{2}\right)\\
 \times \hyper{3}{2}{-n_{1},  n_{1}+2n_{2}+\alpha_{1}+\alpha_{2}+\alpha_{3}+2  ,\ n_{2}+a_{2}+a_{3}+\frac{i\xi_{1}}{2}}{2n_{2}+\alpha_{2}+\alpha_{3}+2,\ n_{2}+a_{1}+a_{2}+a_{3}}{1} ,
\end{multline*}
\begin{equation*}
\Lambda_{2}^{2}(\mathbf{a},\mathbf{\alpha},\mathbf{n};\xi_{2})=B\left(  a_{2}-\frac{i\xi_{2}}{2},a_{3}+\frac{i\xi_{2}}{2}\right) 
 \hyper{3}{2}{-n_{2},n_{2}+\alpha_{2}+\alpha_{3}+1,a_{3}+\frac{i\xi_{2}}{2}}{\alpha_{3}+1,\ a_{2}+a_{3}}{1},
\end{equation*}
or, in terms of the continuous Hahn polynomials \eqref{hahn},
\begin{align*}
\Lambda_{1}^{2}(\mathbf{a},\mathbf{\alpha},\mathbf{n};\xi_{1})&= \frac{n_{1}!B\left( a_{1}-\frac{i\xi _{1}}{2},n_{2}+a_{2}+a_{3}+\frac{%
i\xi _{1}}{2}\right) }{i^{n_{1}}\left( n_{2}+a_{1}+a_{2}+a_{3}\right)
_{n_{1}}\left( 2n_{2}+\alpha _{2}+\alpha _{3}+2\right) _{n_{1}}} \\
& \times p_{n_{1}}\left( \frac{\xi _{1}}{2};n_{2}+a_{2}+a_{3},\alpha
_{1}-a_{1}+1,n_{2}+\alpha _{2}+\alpha _{3}-a_{2}-a_{3}+2,a_{1}\right) ,\\
\Lambda_{2}^{2}(\mathbf{a},\mathbf{\alpha},\mathbf{n};\xi_{2})&=\frac{n_{2}!B\left( a_{2}-\frac{i\xi _{2}}{2},a_{3}+\frac{i\xi _{2}}{2}%
\right) }{i^{n_{2}}\left( a_{2}+a_{3}\right) _{n_{2}}\left( \alpha
_{3}+1\right) _{n_{2}}}  p_{n_{2}}\left( \frac{\xi _{2}}{2};a_{3},\alpha _{2}-a_{2}+1,\alpha
_{3}-a_{3}+1,a_{2}\right),
\end{align*}
where $\mathbf{a}=\left( a_{1},a_{2},a_{3}\right)$, $\mathbf{\alpha}=\left( \alpha_{1},\alpha_{2},\alpha_{3}\right)$, $\mathbf{n}=\left( n_{1},n_{2}\right)$.

\subsection{The Fourier transform of $r$-dimensional polynomials on $r$-simplex}

Let us consider
\begin{multline} \label{18}
g_{r}\left(  x_{1},\dots,x_{r};n_{1},\dots,n_{r},a_{1},\dots,a_{r+1},\alpha_{1},\dots,\alpha_{r+1}\right) \\
=\prod \limits_{j=1}^{r}  \left(  1+\tanh x_{j}\right)^{a_{j}}\left(  1-\tanh x_{j}\right)  ^ { \vert \mathbf{a}^{j+1} \vert}
P_{\mathbf{n}}^{\left(  \mathbf{\alpha}\right)}\left(\Upsilon_{1},\Upsilon_{2},  \dots, \Upsilon_{r} \right)  ,
\end{multline}
for $r\geq1$, where
\begin{align*}
&\Upsilon_{1}(x_{1})=\Upsilon_{1}= \frac{  1+\tanh x_{1}}{2}, \\
&\Upsilon_{r}(x_{1},\dots,x_{r})=\Upsilon_{r}= \frac{\left(  1-\tanh x_{1}\right)  \left(  1-\tanh x_{2}\right)  \cdots \left(  1-\tanh x_{r-1}\right)  \left(  1+\tanh x_{r}\right)}{2^{r}},
\end{align*}
for $r\geq2$ and
\[
\mathbf{n}=\left(  n_{1},\dots,n_{r}\right)  ,\  \mathbf{\alpha}=\left(\alpha_{1},\dots,\alpha_{r+1}\right), \quad  \text{ and}\  \left \vert \mathbf{a}^{j}\right \vert =a_{j}+\cdots+a_{r+1}.
\]
{}From the latter expression, we can write $g_{r+1}$ in terms of $g_{r}$
\begin{multline}\label{3}
g_{r+1}\left(  x_{1},\dots,x_{r+1};n_{1},\dots,n_{r+1},a_{1},\dots,a_{r+2},\alpha_{1},\dots,\alpha_{r+2}\right)  \\
 =2^{-rn_{r+1}}\left(  1+\tanh x_{r+1}\right)  ^{a_{r+1}}\left(  1-\tanh x_{r+1}\right)  ^{a_{r+2}}P_{n_{r+1}}^{\left(  \alpha_{r+2},\alpha_{r+1}\right)  }\left(  \tanh x_{r+1}\right) \\
 \times g_{r}\left(  x_{1},\dots,x_{r};n_{1},\dots,n_{r},a_{1},\dots,a_{r},a_{r+1}+a_{r+2}+n_{r+1},\alpha_{1} \right. \\
 \left. ,\dots,\alpha_{r},\alpha_{r+1}+\alpha_{r+2}+2n_{r+1}+1\right),
\end{multline}
where $P_{n}^{\left(  \alpha_{1},\alpha_{2}\right)  }\left(  x\right)  $ stands for classical univariate Jacobi polynomials defined in \eqref{A:jac}.

In the previous sections, we have already proved the Fourier transforms for $r=1$ and $r=2$. By using induction method, we can calculate the Fourier transform of the function $g_{r}$. We first start with the following theorem.

\begin{theorem}
The following result holds true
\begin{multline} \label{4}
\mathcal{F}\left(  g_{r+1}\left(  x_{1},\dots,x_{r+1};n_{1},\dots,n_{r+1},a_{1},\dots,a_{r+2},\alpha_{1},\dots,\alpha_{r+2}\right)  \right) \\
 =2^{a_{r+1}+a_{r+2}-rn_{r+1}-1}\frac{\left(  \alpha_{r+2}+1\right)_{n_{r+1}} }{n_{r+1}!}
 B\left(  a_{r+1}-\frac{i\xi_{r+1}}{2},a_{r+2}+\frac{i\xi_{r+1}}{2}\right)   \\
 \times \hyper{3}{2}{-n_{r+1},\   n_{r+1}+\alpha_{r+1}+\alpha_{r+2}+1,\ a_{r+2}+\frac{i\xi_{r+1}}{2}}{\alpha_{r+2}+1,\   a_{r+1}+a_{r+2} }{1}   \\
 \times \mathcal{F} \left(  g_{r}\left(  x_{1},\dots,x_{r};n_{1},\dots,n_{r},a_{1},\dots,a_{r},a_{r+1}+a_{r+2}+n_{r+1}\right.  \right.   \\
 \left. \left.  ,\alpha_{1},\dots,\alpha_{r},\alpha_{r+1}+\alpha_{r+2}+2n_{r+1}+1\right)  \right)  .
\end{multline}

\end{theorem}

\begin{proof}
By using the equation \eqref{3}, the Fourier transform of the function $g_{r+1}$ is as follows
\begin{multline*}
\mathcal{F}\left(  g_{r+1}\left(  x_{1},\dots,x_{r+1};n_{1},\dots,n_{r+1},a_{1},\dots,a_{r+2},\alpha_{1},\dots,\alpha_{r+2}\right)  \right)  \\
 =\mathcal{F}\left( 2^{-rn_{r+1}} \left(  1+\tanh x_{r+1}\right)  ^{a_{r+1}}\left(  1-\tanh x_{r+1}\right)  ^{a_{r+2}}P_{n_{r+1}}^{\left(  \alpha_{r+2},\alpha_{r+1}\right)  }\left(  \tanh x_{r+1}\right) \right. \\
 \left.  \times  g_{r}\left(  x_{1},\dots,x_{r};n_{1},\dots,n_{r},a_{1},\dots,a_{r},a_{r+1}+a_{r+2}+n_{r+1},\alpha_{1},\dots \right. \right. \\
 \left. \left. ,\alpha_{r},\alpha_{r+1}+\alpha_{r+2}+2n_{r+1}+1\right)  \right)
 \end{multline*}
\begin{multline*}
 =2^{-rn_{r+1}}\int \limits_{-\infty}^{\infty}\cdots\int \limits_{-\infty}^{\infty}e^{-i\left(  \xi_{1}x_{1}+\dots+\xi_{r}x_{r}+\xi_{r+1}x_{r+1}\right)}\left(  1+\tanh x_{r+1}\right)  ^{a_{r+1}}\left(  1-\tanh x_{r+1}\right)
^{a_{r+2}} \\
 \times P_{n_{r+1}}^{\left(  \alpha_{r+2},\alpha_{r+1}\right)  }\left(  \tanh x_{r+1}\right)  g_{r}\left(x_{1},\dots,x_{r};n_{1},\dots,n_{r},a_{1},\dots\right.  \\
 \left.  ,a_{r},a_{r+1}+a_{r+2}+n_{r+1},\alpha_{1},\dots,\alpha_{r},\alpha_{r+1}+\alpha_{r+2}+2n_{r+1}+1\right)  dx_{r+1}dx_{r}\cdots dx_{1}
 \end{multline*}
\begin{multline*}
 =2^{-rn_{r+1}}\mathcal{F}\left(  g_{r}\left(  x_{1},\dots,x_{r};n_{1},\dots,n_{r},a_{1},\dots,a_{r},a_{r+1}+a_{r+2}+n_{r+1}\right.  \right.  \\
 \left. \left.  ,\alpha_{1},\dots,\alpha_{r},\alpha_{r+1}+\alpha_{r+2}+2n_{r+1}+1\right)  \right) \\
 \times \int \limits_{-\infty}^{\infty}e^{-i\xi_{r+1} x_{r+1}}\left(  1+\tanh x_{r+1}\right)^{a_{r+1}}\left(  1-\tanh x_{r+1}\right)  ^{a_{r+2}} P_{n_{r+1}}^{\left(\alpha_{r+2},\alpha_{r+1}\right)  }\left(  \tanh x_{r+1}\right)  dx_{r+1}
\end{multline*}
\begin{multline*}
 =2^{a_{r+1}+a_{r+2}-rn_{r+1}-1}\mathcal{F} \left(  g_{r}\left(  x_{1},\dots,x_{r};n_{1},\dots,n_{r},a_{1},\dots,a_{r},a_{r+1}+a_{r+2}+n_{r+1}\right.  \right.  \\
  \left.\left.  ,\alpha_{1},\dots,\alpha_{r},\alpha_{r+1}+\alpha_{r+2}+2n_{r+1}+1\right)  \right) \\
 \times \int \limits_{0}^{1} t  ^{a_{r+1}-\frac{i\xi_{r+1}}{2}-1}\left(  1-t\right)  ^{a_{r+2}+\frac{i\xi_{r+1}}{2}-1}P_{n_{r+1}}^{\left(\alpha_{r+2},\alpha_{r+1}\right)  }\left(  2t-1\right)  dt
\end{multline*}%
\begin{multline*}
 =2^{a_{r+1}+a_{r+2}-rn_{r+1}-1}\mathcal{F} \left(  g_{r}\left(  x_{1},\dots,x_{r};n_{1},\dots,n_{r},a_{1},\dots,a_{r},a_{r+1}+a_{r+2}+n_{r+1}\right.  \right.  \\
  \left.\left.  ,\alpha_{1},\dots,\alpha_{r},\alpha_{r+1}+\alpha_{r+2}+2n_{r+1}+1\right)  \right) \\
 \times \int \limits_{0}^{1} t  ^{a_{r+1}-\frac{i\xi_{r+1}}{2}-1}\left(  1-t\right)  ^{a_{r+2}+\frac{i\xi_{r+1}}{2}-1} \frac{\left(  \alpha_{r+2}+1\right)_{n_{r+1}}  }{ n_{r+1}! }\\
\times
\hyper{2}{1}{-n_{r+1}, n_{r+1}+\alpha_{r+1}+\alpha_{r+2}+1}{\alpha_{r+2}+1 }{1-t}  dt
\end{multline*}%
\begin{multline*}
 = \frac{2^{a_{r+1}+a_{r+2}-rn_{r+1}-1}\left(  \alpha_{r+2}+1\right)_{n_{r+1}}  }{ n_{r+1}! }
\mathcal{F}\left(  g_{r}\left(  x_{1},\dots,x_{r};n_{1},\dots,n_{r}\right.  \right.  \\
 \left.  \left.  ,a_{1},..,a_{r},a_{r+1}+a_{r+2}+n_{r+1},\alpha_{1},\dots,\alpha_{r},\alpha_{r+1}+\alpha_{r+2}+2n_{r+1}+1\right)  \right)  \\
 \times \sum \limits_{l=0}^{n_{r+1}}\frac{\left(  -n_{r+1}\right)  _{l}\left(n_{r+1}+\alpha_{r+1}+\alpha_{r+2}+1\right)  _{l}}{\left(  \alpha_{r+2}+1\right)_{l} l!} \int \limits_{0}^{1} t  ^{a_{r+1}-\frac{i\xi_{r+1}}{2}-1}\left(  1-t\right)^{a_{r+2}+\frac{i\xi_{r+1}}{2}-1+l}dt
\end{multline*}
\begin{multline*}
= \frac{2^{a_{r+1}+a_{r+2}-rn_{r+1}-1}\left(  \alpha_{r+2}+1\right)_{n_{r+1}}  }{ n_{r+1}! } B\left(  a_{r+1}-\frac{i\xi_{r+1}}{2},\ a_{r+2}+\frac{i\xi_{r+1}}{2}\right)
\\
\times \hyper{3}{2}{-n_{r+1},\  n_{r+1}+\alpha_{r+1}+\alpha_{r+2}+1,\ a_{r+2}+\frac{i\xi_{r+1}}{2}}{\alpha_{r+2}+1,\  a_{r+1}+a_{r+2}}{1}\\
\times \mathcal{F}\left(  g_{r}\left(  x_{1},\dots,x_{r};n_{1},\dots,n_{r},a_{1},..,a_{r},a_{r+1}+a_{r+2}+n_{r+1}\right.  \right.  \\
\left.  \left. ,\alpha_{1},\dots,\alpha_{r},\alpha_{r+1}+\alpha_{r+2}+2n_{r+1}+1\right)  \right),
\end{multline*}
which proves the theorem.
\end{proof}

\bigskip

By applying Theorem 3.1 successively, the following theorem can be given.

\begin{theorem}
By substituting Fourier transform of the specific function in the right-hand
side of the equality (\ref{4}), we have%
\begin{multline}\label{A}
\mathcal{F}
\left(  g_{r}\left(  \boldsymbol{x};\mathbf{n},\mathbf{a},\mathbf{\alpha} \right)  \right) \\
=2^{r\left({a}_{r}+a_{r+1}-1 \right )
+\sum \limits_{j=1}^{r-1}ja_{j}} \prod \limits_{j=1}^{r}\left \{  \frac{\left( 2\left \vert
\mathbf{n}^{j+1}\right \vert+\left \vert
\mathbf{\alpha}^{j+1}\right \vert+r-j+1\right)_{n_{j}}  }{n_{j}!} \Lambda_{j}^{r}(\mathbf{a},\mathbf{\alpha},\mathbf{n};\xi_{j})
\right \},
\end{multline}
where
\begin{multline*}
\Lambda_{j}^{r}(\mathbf{a},\mathbf{\alpha},\mathbf{n};\xi_{j})=B\left(
a_{j}-\frac{i\xi_{j}}{2},\  \left \vert \mathbf{n}^{j+1}\right \vert +\left \vert
\mathbf{a}^{j+1}\right \vert +\frac{i\xi_{j}}{2}\right)  \\
\times \hyper{3}{2}{-n_{j},\ n_{j}+2\left \vert \mathbf{n}^{j+1}\right \vert +\left \vert \mathbf{\alpha}^{j}\right \vert +r-j+1,\ \left \vert \mathbf{n}^{j+1}\right \vert +\left \vert \mathbf{a}^{j+1}\right \vert+\frac{i\xi_{j}}{2}}{2\left \vert \mathbf{n}^{j+1}\right \vert +\left \vert \mathbf{\alpha}^{j+1}\right \vert +r-j+1,\ \left \vert \mathbf{n}^{j+1}\right \vert +\left \vert \mathbf{a}^{j}\right \vert }{1},
\end{multline*}
or, in terms of the continuous Hahn polynomials \eqref{hahn}
\begin{multline*}
\Lambda_{j}^{r}(\mathbf{a},\mathbf{\alpha},\mathbf{n};\xi_{j})=
 \frac{n_{j}!B\left(
a_{j}-\frac{i\xi_{j}}{2},\  \left \vert \mathbf{n}^{j+1}\right \vert +\left \vert
\mathbf{a}^{j+1}\right \vert +\frac{i\xi_{j}}{2}\right)}{i^{n_{j}}\left( \left \vert \mathbf{n}^{j+1}\right \vert
+\left \vert \mathbf{a}^{j}\right \vert \right) _{n_{j}}\left( 2\left \vert
\mathbf{n}^{j+1}\right \vert +\left \vert \alpha ^{j+1}\right \vert
+r-j+1\right) _{n_{j}}} \\
\times p_{nj}\left( \frac{\xi _{j}}{2};\left \vert \mathbf{n}^{j+1}\right \vert
+\left \vert \mathbf{a}^{j+1}\right \vert ,\alpha _{j}-a_{j}+1,\left \vert
\mathbf{n}^{j+1}\right \vert +\left \vert \alpha ^{j+1}\right \vert -\left \vert
\mathbf{a}^{j+1}\right \vert +r-j+1,a_{j}\right),
\end{multline*}
 $ g_{r}\left(  \boldsymbol{x};\mathbf{n},\mathbf{a},\mathbf{\alpha} \right) $ is the function given in \eqref{3} for $r\geq1$, and
\begin{gather*}
\pmb{x}  \boldsymbol{=}\left(  x_{1},\dots,x_{r}\right), \quad \mathbf{n}=\left(  n_{1},\dots,n_{r}\right), \quad \mathbf{n}^{j}=\left(  n_{j},\dots,n_{r}\right), \\ \left \vert \mathbf{n}^{j}\right \vert=n_{j}+\cdots+n_{r},\quad
\mathbf{\alpha}  =\left(  \alpha_{1},\dots,\alpha_{r+1}\right),\quad  \mathbf{\alpha}^{j}=\left(  \alpha_{j},\dots,\alpha_{r+1}\right), \\
\mathbf{a}=\left(  a_{1},\dots,a_{r+1}\right), \mathbf{a}^{j}=\left(  a_{j},\dots,a_{r+1}\right), \quad \left \vert \mathbf{a}^{j}\right \vert  =a_{j}+\cdots+a_{r+1}.
\end{gather*}
\end{theorem}
\begin{proof}
 In subsection 3.1 and 3.2, we show that for the cases $r=1$ and $r=2$ this result is satisfied. By using induction method, when we apply Theorem 3.1 successively, it is satisfied for each $r$.
 \end{proof}

\section{Some classes of special functions using Fourier transforms of orthogonal polynomials on the simplex}\label{sec:ops}

The Parseval's identity corresponding to \eqref{16}, is given by the statement (cf. \cite{MR0942661})
\[
\int \limits_{-\infty}^{\infty}g\left(  x\right)  \overline{h\left(  x\right)}dx=\frac{1}{2\pi}\int \limits_{-\infty}^{\infty}\mathcal{F}\left(  g\left(  x\right)  \right)  \overline{\mathcal{F}\left(  h\left(  x\right)  \right)  }d\xi,
\]
and the Parseval's identity corresponding to \eqref{17}, is given by (cf. \cite{MR1867914})
\begin{multline}\label{19}
 \int \limits_{-\infty}^{\infty}\cdots\int \limits_{-\infty}^{\infty}g\left(x_{1},\dots,x_{r}\right)  \overline{h\left(  x_{1},\dots,x_{r}\right)  }dx_{1}\dots dx_{r}\\
  =\frac{1}{\left(  2\pi \right)  ^{r}}\int \limits_{-\infty}^{\infty}\dots\int \limits_{-\infty}^{\infty}\mathcal{F}\left(  g\left(  x_{1},\dots,x_{r}\right)  \right)  \overline{\mathcal{F}\left(  h\left(  x_{1},\dots,x_{r}\right)  \right)  }d\xi_{1}\cdots d\xi_{r}.
\end{multline}

\subsection{The class of special functions using Fourier transform of Jacobi polynomials on the 1-simplex}

By use of \eqref{55} and \eqref{99} in Parseval's identity, Koelink \cite{MR1307541} proved that the special function%
\begin{multline}\label{1boyut}
\ _{1}S_{n}\left( x;a_{1},a_{2},b_{1},b_{2}\right) =\hyper{3}{2}{-n ,\
a_{2}+x/2,\ n+a_{1}+a_{2}+b_{1}+b_{2}-1}{a_{2}+b_{2},\ a_{1}+a_{2}}{1}
\\
\\
=\frac{n!i^{-n}}{\left( a_{1}+a_{2}\right) _{n}\left( a_{2}+b_{2}\right) _{n}%
}p_{n}\left( \frac{-ix}{2};a_{2},b_{1},b_{2},a_{1}\right)
\end{multline}
has an orthogonality relation of the form%
\begin{multline*}
\int \limits_{-\infty }^{\infty }\Gamma \left( a_{1}-\frac{ix}{2}\right)
\Gamma \left( a_{2}+\frac{ix}{2}\right) \Gamma \left( b_{1}+\frac{ix}{2}%
\right) \Gamma \left( b_{2}-\frac{ix}{2}\right)  \\
\times \ _{1}S_{n}\left( ix;a_{1},a_{2},b_{1},b_{2}\right) \ _{1}S_{m}\left(
-ix;b_{1},b_{2},a_{1},a_{2}\right) dx \\
=\frac{2^{2}\pi n!\Gamma \left( n+a_{1}+b_{1}\right) \Gamma ^{2}\left(
a_{2}+b_{2}\right) \Gamma \left( a_{1}+a_{2}\right) }{\left(
2n+a_{1}+a_{2}+b_{1}+b_{2}-1\right) \Gamma \left( n+a_{2}+b_{2}\right) }  \frac{\Gamma \left( b_{1}+b_{2}\right) }{\Gamma \left(
n+a_{1}+a_{2}+b_{1}+b_{2}-1\right) }\delta _{n,m}\\
= \frac{2^{2}\pi \left(
n!\right) ^{2}~\Gamma \left( a_{1}+a_{2}\right) \Gamma \left(
b_{1}+b_{2}\right) }{\left( \left( a_{2}+b_{2}\right) _{n}\right) ^{2}}%
h_{n}^{\left( a_{1}+b_{1}-1,a_{2}+b_{2}-1\right) } \delta _{n,m},
\end{multline*}%
where $h_{n}^{\left( a_{1}+b_{1}-1,a_{2}+b_{2}-1\right) }$ is defined as \eqref{B}, from which, it follows
\begin{multline*}
\int \limits_{-\infty }^{\infty }\Gamma \left( a_{2}+ix\right) \Gamma \left(
a_{1}-ix\right) \Gamma \left( b_{1}+ix\right) \Gamma \left( b_{2}-ix\right)
\\
\times p_{n}\left( x;\ a_{2},\ b_{1},\ b_{2},\ a_{1}\right) p_{m}\left( x;\
a_{2},\ b_{1},\ b_{2},\ a_{1}\right) dx \\
=\frac{2\pi \Gamma \left( n+a_{2}+b_{2}\right) \Gamma \left(
n+a_{1}+b_{1}\right) \Gamma \left( n+a_{1}+a_{2}\right) \Gamma \left(
n+b_{1}+b_{2}\right) }{n!\left( 2n+a_{1}+a_{2}+b_{1}+b_{2}-1\right) \Gamma
\left( n+a_{1}+a_{2}+b_{1}+b_{2}-1\right) }\delta _{n,m},
\end{multline*}%
for $a_{1}$, $a_{2}$, $b_{1}$, $b_{2}>0$.

\subsection{The class of special functions using Fourier transform of two-dimensional polynomials on the 2-simplex}

By substituting \eqref{5} and \eqref{1} in Parseval's identity we get
\begin{multline*}
 \int \limits_{-\infty}^{\infty}\int \limits_{-\infty}^{\infty}g_{2}\left(x_{1},x_{2};n_{1},n_{2},a_{1},a_{2},a_{3},\alpha_{1},\alpha_{2},\alpha_{3}\right)  \\
 \times g_{2}\left(  x_{1},x_{2};m_{1},m_{2},b_{1},b_{2},b_{3},\beta
_{1},\beta_{2},\beta_{3}\right)  dx_{1}dx_{2}
\end{multline*}
\begin{multline*}
 =\int \limits_{-\infty}^{\infty}\int \limits_{-\infty}^{\infty}\left(  1+\tanh x_{1}\right)  ^{a_{1}+b_{1}}\left(  1-\tanh x_{1}\right)  ^{a_{2}+a_{3}+b_{2}+b_{3}}\left(  1+\tanh x_{2}\right)  ^{a_{2}+b_{2}} \\
 \times \left(  1-\tanh x_{2}\right)  ^{a_{3}+b_{3}} P_{n_{1},n_{2}}^{\left(  \alpha_{1},\alpha_{2},\alpha_{3}\right)}\left( \Upsilon_{1}(x_{1}),\Upsilon_{2}(x_{1},x_{2}) \right) \\
\times  P_{m_{1},m_{2}}^{\left(  \beta_{1},\beta_{2},\beta_{3}\right)}\left(  \Upsilon_{1}(x_{1}),\Upsilon_{2}(x_{1},x_{2})  \right)  dx_{1}dx_{2}
 \end{multline*}
\begin{multline*}
 =\frac{1}{\left(  2\pi \right)  ^{2}}\int \limits_{-\infty}^{\infty}\int \limits_{-\infty}^{\infty}\frac{\left( 2n_{2}+\alpha_{2}+\alpha_{3}+2\right)_{n_{1}}\left( 2m_{2}+\beta_{2}+\beta_{3}+2\right)_{m_{1}} \left( \alpha_{3}+1\right)_{n_{2}}}{
2^{4-a_{1}-b_{1}-2\left(  a_{2}+a_{3}+b_{2}+b_{3}\right)}
 n_{1}!n_{2}!m_{1}!m_{2}! }\\
 \times \left(\beta_{3}+1\right)_{m_{2}} \Lambda_{1}^{2}(\mathbf{a},\mathbf{\alpha},\mathbf{n};\xi_{1}) \overline{\Lambda_{1}^{2}(\mathbf{b},\mathbf{\beta},\mathbf{m};\xi_{1}) }
 \Lambda_{2}^{2}(\mathbf{a},\mathbf{\alpha},\mathbf{n};\xi_{2}) \overline{\Lambda_{2}^{2}(\mathbf{b},\mathbf{\beta},\mathbf{m};\xi_{2}) }
 d\xi_{1}d\xi_{2}
 \end{multline*}
 where  $\mathbf{n}=\left(  n_{1}, n_{2}\right)$, $\mathbf{m}=\left(  m_{1}, m_{2}\right)$, $\mathbf{\alpha}=\left( \alpha_{1},\alpha_{2},\alpha_{3}\right)$, $\mathbf{\beta}=\left( \beta_{1},\beta_{2},\beta_{3}\right)$, $\mathbf{a}=\left(  a_{1},a_{2},a_{3}\right)$ and $\mathbf{b}=\left(  b_{1},b_{2},b_{3}\right)$.

Now, using the transforms $\tanh x_{1}=2u-1$ and $\tanh x_{2}=\frac{2v}{1-u}-1$ in the left-hand side of the equality, respectively, we have
\begin{multline}\label{6}
 2^{4}\pi^{2}\int \limits_{0}^{1}\int \limits_{0}^{1-u} u^{a_{1}+b_{1}-1} v^{a_{2}+b_{2}-1}\left(  1-u-v\right)  ^{a_{3}+b_{3}-1}P_{n_{1},n_{2}}^{\left(\alpha_{1},\alpha_{2},\alpha_{3}\right)  }\left(  u,v\right)  P_{m_{1},m_{2}%
}^{\left(  \beta_{1},\beta_{2},\beta_{3}\right)  }\left(  u,v\right)dvdu\\
=\int \limits_{-\infty}^{\infty}\int \limits_{-\infty}^{\infty}\frac{ \left(  2n_{2}+\alpha_{2}+\alpha_{3}+2\right)_{n_{1}} \left(  2m_{2}+\beta_{2}+\beta_{3}+2\right)_{m_{1}} \left( \alpha_{3}+1\right)_{n_{2}} \left( \beta_{3}+1\right)_{m_{2}}}{n_{1}!n_{2}!m_{1}!m_{2}!}\\
\times
\Lambda_{1}^{2}(\mathbf{a},\mathbf{\alpha},\mathbf{n};\xi_{1})\overline{\Lambda_{1}^{2}(\mathbf{b},\mathbf{\beta},\mathbf{m};\xi_{1}) } \Lambda_{2}^{2}(\mathbf{a},\mathbf{\alpha},\mathbf{n};\xi_{2}) \overline{\Lambda_{2}^{2}(\mathbf{b},\mathbf{\beta},\mathbf{m};\xi_{2}) } d\xi_{1}d\xi_{2}.
\end{multline}
On the other hand, if in the left-hand side of (\ref{6}) we take $a_{1}+b_{1}-1=\alpha_{1}=\beta_{1}$, $a_{2}+b_{2}-1=\alpha_{2}=\beta_{2}$ and $a_{3}+b_{3}-1=\alpha_{3}=\beta_{3}$ then according to the orthogonality relation \eqref{A:3}, equation \eqref{6} reduces to
\begin{multline*}
2^{4}\pi^{2}\int \limits_{0}^{1}\int \limits_{0}^{1-u} u^{a_{1}+b_{1}-1} v^{a_{2}+b_{2}-1}\left(  1-u-v\right)  ^{a_{3}+b_{3}-1}\\
\times P_{n_{1},n_{2}}^{\left(a_{1}+b_{1}-1,a_{2}+b_{2}-1,a_{3}+b_{3}-1\right)  }\left(  u,v\right)  P_{m_{1},m_{2}%
}^{\left( a_{1}+b_{1}-1,a_{2}+b_{2}-1,a_{3}+b_{3}-1\right)  }\left(  u,v\right)dvdu\\
=2^{4}\pi^{2}h_{n_{1},n_{2}}^{\left(  a_{1}+b_{1}-1,a_{2}+b_{2}-1,a_{3}+b_{3}-1\right)  }\delta_{n_{1},m_{1}}\delta_{n_{2},m_{2}}
\end{multline*}
\begin{multline*}
=\int \limits_{-\infty}^{\infty}\int \limits_{-\infty}^{\infty}\frac{ \left(  2n_{2}+\alpha_{2}+\alpha_{3}+2\right)_{n_{1}} \left(  2m_{2}+\beta_{2}+\beta_{3}+2\right)_{m_{1}} \left( \alpha_{3}+1\right)_{n_{2}} \left( \beta_{3}+1\right)_{m_{2}}}{n_{1}!n_{2}!m_{1}!m_{2}!}\\
\times
\Lambda_{1}^{2}(\mathbf{a},\mathbf{\alpha},\mathbf{n};\xi_{1})\overline{\Lambda_{1}^{2}(\mathbf{b},\mathbf{\beta},\mathbf{m};\xi_{1}) } \Lambda_{2}^{2}(\mathbf{a},\mathbf{\alpha},\mathbf{n};\xi_{2}) \overline{\Lambda_{2}^{2}(\mathbf{b},\mathbf{\beta},\mathbf{m};\xi_{2}) } d\xi_{1}d\xi_{2},
\end{multline*}
so
\begin{multline*}
 \frac{2^{4}\pi^{2}\left(  n_{1}!\right)  ^{2}\left(  n_{2}!\right)  ^{2}\Gamma\left(  n_{2}+a_{1}+a_{2}+a_{3}\right)  \Gamma\left(  n_{2}+b_{1}+b_{2}+b_{3}\right)\Gamma\left(  a_{2}+a_{3}\right)  \Gamma\left(  b_{2}+b_{3}\right)}{\left(  2n_{2}+a_{2}+b_{2}+a_{3}+b_{3}\right)_{n_{1}}^{2}  \left(  a_{3}+b_{3}\right)_{n_{2}}^{2} }\\
\times h_{n_{1},n_{2}}^{\left(  a_{1}+b_{1}-1,a_{2}+b_{2}-1,a_{3}+b_{3}-1\right)  }\delta_{n_{1},m_{1}}\delta_{n_{2},m_{2}}%
\end{multline*}
\begin{multline*}
 =\int \limits_{-\infty}^{\infty}\int \limits_{-\infty}^{\infty}\Theta_{2}\left(  \xi_{1},\xi_{2};a_{1},a_{2},a_{3}\right)  \Lambda_{2}\left(\xi_{1},\xi_{2};n_{1},n_{2},a_{1},a_{2},a_{3},b_{1},b_{2},b_{3}\right)  \\
\times \overline{\Theta_{2}\left(  \xi_{1},\xi_{2};b_{1},b_{2},b_{3}\right)\Lambda_{2}\left(  \xi_{1},\xi_{2};m_{1},m_{2},b_{1},b_{2},b_{3},a_{1},a_{2},a_{3}\right)  }d\xi_{1}d\xi_{2}
\end{multline*}
where
\begin{multline*}
\Theta_{2}\left(  \xi_{1},\xi_{2};a_{1},a_{2},a_{3}\right)\\ =\Gamma \left(a_{1}-\frac{i\xi_{1}}{2}\right)  \Gamma \left(  a_{2}+a_{3}+\frac{i\xi_{1}}{2}\right)  \Gamma \left(  a_{2}-\frac{i\xi_{2}}{2}\right) \Gamma \left(  a_{3}+\frac{i\xi_{2}}{2}\right)  \left(  a_{2}+a_{3}+\frac{i\xi_{1}}{2}\right)_{n_{2}},
\end{multline*}%
\begin{multline*}
\Lambda_{2}\left(  \xi_{1},\xi_{2};n_{1},n_{2},a_{1},a_{2},a_{3},b_{1},b_{2},b_{3}\right)  \\
=\hyper{3}{2}{-n_{2},\  n_{2}+a_{2}+b_{2}+a_{3}+b_{3}-1,\ a_{3}+\frac{i\xi_{2}}{2}}{a_{3}+b_{3},\   a_{2}+a_{3} }{1}      \\
\times \hyper{3}{2}{-n_{1},\ n_{1}+2n_{2}+a_{1}+b_{1}+a_{2}+b_{2}+a_{3}+b_{3}-1 ,\ n_{2}+a_{2}+a_{3}+\frac{i\xi_{1}}{2}}{2n_{2}+a_{2}+b_{2}+a_{3}+b_{3},\ n_{2}+a_{1}+a_{2}+a_{3}}{1},
\end{multline*}
or, in terms of the Hahn polynomials \eqref{hahn}
\begin{multline*}
\Lambda_{2}\left(  \xi_{1},\xi_{2};n_{1},n_{2},a_{1},a_{2},a_{3},b_{1},b_{2},b_{3}\right)  \\
=\frac{n_{1}!n_{2}!i^{-n_{1}-n_{2}}}{\left( a_{2}+a_{3}\right)
_{n_{2}}\left( a_{3}+b_{3}\right) _{n_{2}}\left(
n_{2}+a_{1}+a_{2}+a_{3}\right) _{n_{1}}\left(
2n_{2}+a_{2}+b_{2}+a_{3}+b_{3}\right) _{n_{1}}} \\
 \times p_{n_{1}}\left( \frac{\xi _{1}}{2}%
;n_{2}+a_{2}+a_{3},b_{1},n_{2}+b_{2}+b_{3},a_{1}\right)
 p_{n_{2}}\left( \frac{\xi _{2}}{2};a_{3},b_{2},b_{3},a_{2}\right),
\end{multline*}
and $h_{n_{1},n_{2}}^{\left(  \alpha_{1},\alpha_{2},\alpha_{3}\right)  }$ is given by \eqref{B}. As a result, the following theorem can be given.

\begin{theorem}
The families of functions
\begin{multline*}
\ _{2}S_{n_{1},n_{2}}\left(  x_{1},x_{2};a_{1},a_{2},a_{3},b_{1},b_{2},b_{3}\right)  \\
=\left(  a_{2}+a_{3}+\frac{x_{1}}{2}\right)_{n_{2}}\Lambda_{2}\left(  -ix_{1},-ix_{2};n_{1},n_{2},a_{1},a_{2},a_{3},b_{1},b_{2},b_{3}\right)
\end{multline*}
are orthogonal with respect to the weight function
\begin{multline*}
W_{2}\left(  x_{1},x_{2},a_{1},a_{2},a_{3},b_{1},b_{2},b_{3}\right)  \\
=\Gamma \left(  a_{1}-\frac{ix_{1}}{2}\right)  \Gamma \left(  a_{2}+a_{3}+\frac{ix_{1}}{2}\right)  \Gamma \left(  a_{2}-\frac{ix_{2}}{2}\right)  \Gamma \left(  a_{3}+\frac{ix_{2}}{2}\right)  \\
\times \Gamma \left(  b_{1}+\frac{ix_{1}}{2}\right)  \Gamma \left(b_{2}+b_{3}-\frac{ix_{1}}{2}\right)  \Gamma \left(  b_{2}+\frac{ix_{2}}{2}\right)  \Gamma \left(  b_{3}-\frac{ix_{2}}{2}\right)  ,
\end{multline*}
and the corresponding orthogonality relation is
\begin{multline*}
 \int \limits_{-\infty}^{\infty}\int \limits_{-\infty}^{\infty}W_{2}\left(x_{1},x_{2},a_{1},a_{2},a_{3},b_{1},b_{2},b_{3}\right)  \ _{2}S_{n_{1},n_{2}}\left(  ix_{1},ix_{2};a_{1},a_{2},a_{3},b_{1},b_{2},b_{3}\right)  \\
 \times \ _{2}S_{m_{1},m_{2}}\left(  -ix_{1},-ix_{2};b_{1},b_{2},b_{3},a_{1},a_{2},a_{3}\right)  dx_{1}dx_{2} \\
 =\frac{2^{4}\pi^{2}\left(  n_{1}!\right)  ^{2}\left(  n_{2}!\right)  ^{2}\Gamma\left(  n_{2}+a_{1}+a_{2}+a_{3}\right)  \Gamma\left(  n_{2}+b_{1}+b_{2}+b_{3}\right)\Gamma\left(  a_{2}+a_{3}\right)  \Gamma\left(  b_{2}+b_{3}\right)}{\left(  2n_{2}+a_{2}+b_{2}+a_{3}+b_{3}\right)_{n_{1}}^{2}  \left(  a_{3}+b_{3}\right)_{n_{2}}^{2} }\\
\times h_{n_{1},n_{2}}^{\left(  a_{1}+b_{1}-1,a_{2}+b_{2}-1,a_{3}+b_{3}-1\right)  }\delta_{n_{1},m_{1}}\delta_{n_{2},m_{2}}%
\end{multline*}
for $a_{1}$, $a_{2}$, $a_{3}$, $b_{1}$, $b_{2}$, $b_{3}>0$ where $h_{n_{1},n_{2}}^{\left(  \alpha_{1},\alpha_{2},\alpha_{3}\right)  }$ is given by \eqref{B}.
\end{theorem}

\begin{remark}
The weight function of this orthogonality relation is positive when the three equalities hold simultaneously: $a_{1}=b_{1}$,  $a_{2}=b_{2}$, and $a_{3}=b_{3}$.
\end{remark}

\subsection{The class of special functions using Fourier transform of $r$-dimensional polynomials on the $r$-simplex}

By applying the method in the cases $r=1$ and $r=2$, if we substitute \eqref{18} and \eqref{A} in the Parseval's identity \eqref{19}, after the necessary calculations we obtain the following theorem.
\begin{theorem}
Let be $\boldsymbol{n}:=\left(  n_{1},n_{2},\dots,n_{r}\right)  $, $\boldsymbol{m}:=\left(  m_{1},m_{2},\dots,m_{r}\right)  $, $\left \vert \mathbf{n}^{j}\right \vert =n_{j}+n_{j+1}+\dots+n_{r}$ and $\boldsymbol{x}
:=\left(  x_{1},x_{2},\dots,x_{r}\right)  $ for $\boldsymbol{x}\in \mathbb{R}^{r}$. In here, $\mathbf{a}:=\left(  a_{1},a_{2},\dots,a_{r+1}\right)  $ and $\mathbf{b}:=\left(  b_{1},b_{2},\dots,b_{r+1}\right).  $ Let
\begin{align*}
\mathbf{a}^{j}  & =\left(  a_{j},\dots,a_{r+1}\right)  ,\  \  \  \ 1\leq j\leq r+1,\\
\mathbf{b}^{j}  & =\left(  b_{j},\dots,b_{r+1}\right)   ,\  \  \  \ 1\leq j\leq r+1,
\end{align*}
so $\left \vert \mathbf{a}^{j}\right \vert =a_{j}+a_{j+1}+\cdots+a_{r+1}$ and $\left \vert \mathbf{b}^{j}\right \vert =b_{j}+b_{j+1}+\cdots+b_{r+1}$. The following equality is satisfied:%
\begin{multline*}
 \int \limits_{-\infty}^{\infty}\cdots\int \limits_{-\infty}^{\infty}W_{r}\left(\boldsymbol{x},\mathbf{a},\mathbf{b}\right)  ~_{r}S_{\mathbf{n}}\left(  i\boldsymbol{x};\mathbf{a},\mathbf{b}\right) ~ _{r}S_{\mathbf{m}}\left(  -i\boldsymbol{x};\mathbf{b},\mathbf{a}\right) \mathbf{dx}\\
=2^{2r}\pi^{r}h_{\mathbf{n}}^{\left(  \mathbf{a+b-1}\right)  } \prod \limits_{j=1}^{r}  \frac{\left(  n_{j}!\right)  ^{2} \Gamma \left(  \left \vert \mathbf{n}%
^{j+1}\right \vert +\left \vert \mathbf{a}^{j}\right \vert \right)\Gamma \left(  \left \vert \mathbf{n}^{j+1}\right \vert +\left \vert \mathbf{b}^{j}\right \vert \right)  }{\left(  2\left \vert \mathbf{n}%
^{j+1}\right \vert  +\left \vert \mathbf{a}%
^{j+1}\right \vert+\left \vert \mathbf{b}%
^{j+1}\right \vert\right)_{n_{j}}^{2} }\delta_{n_{j},m_{j}}
\end{multline*}
where
\begin{multline*}
W_{r}\left(  \boldsymbol{x},\mathbf{a},\mathbf{b}\right)  :=W_{r}\left(  x_{1},\dots,x_{r};a_{1},\dots,a_{r+1},b_{1},\dots,b_{r+1}\right)  \\
 =\prod \limits_{j=1}^{r}\left \{  \Gamma \left(  a_{j}-\frac{ix_{j}}{2}\right)  \Gamma \left(  \left \vert \mathbf{a}^{j+1}\right \vert+\frac{ix_{j}}{2}\right) \Gamma \left(  b_{j}+\frac{ix_{j}}{2}\right)  \Gamma \left(\left \vert \mathbf{b}^{j+1}\right \vert -\frac{ix_{j}}{2}\right) \right \}
\end{multline*}
for $a_{j}, b_{j}>0$ ; $j=1,2,\dots,r+1$ and
\begin{multline}\label{Func1}
_{r}S_{\mathbf{n}}\left(  \boldsymbol{x};\mathbf{a},\mathbf{b}\right) =\prod \limits_{j=1}^{r}\left \{  \left(  \left \vert \mathbf{a}^{j+1}\right \vert +\frac{x_{j}}{2}\right)  _{\left \vert \mathbf{n}^{j+1}\right \vert }\right.  \\
 \left.  \times \hyper{3}{2}{-n_{j},  n_{j}+2\left \vert \mathbf{n}^{j+1}\right \vert +\left \vert \mathbf{a}^{j}\right \vert +\left \vert \mathbf{b}^{j}\right \vert -1  ,\ \left \vert \mathbf{n}^{j+1}\right \vert +\left \vert \mathbf{a}^{j+1}\right \vert+\frac{x_{j}}{2}}
 {2\left \vert \mathbf{n}^{j+1}\right \vert +\left \vert \mathbf{a}^{j+1}\right \vert +\left \vert \mathbf{b}^{j+1}\right \vert,\ \left \vert \mathbf{n}^{j+1}\right \vert +\left \vert \mathbf{a}^{j}\right \vert }
 {1}   \right \}
\end{multline}

or, in terms of Hahn polynomials \eqref{hahn}

\begin{multline*}
_{r}S_{\mathbf{n}}\left(  \boldsymbol{x};\mathbf{a},\mathbf{b}\right) =\prod \limits_{j=1}^{r}\left \{ \frac{n_{j}!i^{-n_{j}}}{\left( \left \vert \mathbf{n}^{j+1}\right \vert
+\left \vert \mathbf{a}^{j}\right \vert \right) _{n_{j}}\left( 2\left \vert
\mathbf{n}^{j+1}\right \vert +\left \vert \mathbf{a}^{j+1}\right \vert
+\left \vert \mathbf{b}^{j+1}\right \vert \right)_{n_{j}}} \right.  \\
\times \left. \left.  \left(  \left \vert \mathbf{a}^{j+1}\right \vert +\frac{x_{j}}{2}\right)  _{\left \vert \mathbf{n}^{j+1}\right \vert }\right.  p_{n_{j}}\left( \frac{-ix_{j}}{2};\left \vert \mathbf{n}^{j+1}\right \vert +\left \vert \mathbf{a}^{j+1}\right \vert ,b_{j},\left \vert
\mathbf{n}^{j+1}\right \vert +\left \vert \mathbf{b}^{j+1}\right \vert,a_{j}\right) \right. \}
\end{multline*}
for $r\geq1$.
\end{theorem}

\begin{remark}
The weight function of this orthogonality relation is positive when all equalities hold simultaneously: $a_{j}=b_{j}$ for $j=1,2,\dots,r+1$.
\end{remark}

\section{Recurrence relations for the functions $_{r}S_{\mathbf{n}}\left( \mathbf{x};\mathbf{a},\mathbf{b}\right) $}\label{sec:fs}

In this section, we derive several recurrence relations for $_{r}S_{\mathbf{n}}\left( \mathbf{x};\mathbf{a},\mathbf{b}\right) $
given by \eqref{Func1}. In doing so, we first recall the well-known relations for
hypergeometric function $_{3}F_{2}$ $\left(
m_{1},m_{2},m_{3};s_{1},s_{2};z\right) $:

\begin{lemma}
For $\left \vert z\right \vert <1$, the hypergeometric function $\ _{3}F_{2}$
satisfies the following known recurrence relations:

\emph{(i)}%
\begin{eqnarray}
\left( z-1\right) \ _{3}F_{2}\left( m_{1}+1,m_{2},m_{3};s_{1},s_{2};z\right)
=\left( B_{1}+C_{1}z\right) \ _{3}F_{2}\left(
m_{1},m_{2},m_{3};s_{1},s_{2};z\right)  &&  \notag \\
+\left( B_{2}+C_{2}z\right) \ _{3}F_{2}\left(
m_{1}-1,m_{2},m_{3};s_{1},s_{2};z\right) +B_{3}\ _{3}F_{2}\left(
m_{1}-2,m_{2},m_{3};s_{1},s_{2};z\right)  &&  \label{H1}
\end{eqnarray}
where%
\begin{gather*}
B_{1} =\frac{s_{1}+s_{2}+1-3m_{1}}{m_{1}},\ B_{2}=\frac{s_{1}s_{2}+\left(
m_{1}-1\right) \left( 3m_{1}-2\left( s_{1}+s_{2}+1\right) \right) }{\left(
m_{1}-1\right) _{2}}, \\
C_{1} =\frac{2m_{1}-m_{2}-m_{3}-1}{m_{1}},\ C_{2}=-\frac{\left(
m_{1}-1\right) \left( m_{1}-m_{2}-m_{3}-1\right) +m_{2}m_{3}}{\left(
m_{1}-1\right) _{2}},
\end{gather*}
and
\begin{equation*}
B_{3}=-\frac{\left( m_{1}-s_{1}-1\right) \left( m_{1}-s_{2}-1\right) }{\left( m_{1}-1\right) _{2}},
\end{equation*}

\emph{(ii)}%
\begin{multline}\label{H2}
\left( z-1\right) \ _{3}F_{2}\left( m_{1},m_{2},m_{3};s_{1}-1,s_{2};z\right)
=\left( B_{1}+C_{1}z\right) \ _{3}F_{2}\left(
m_{1},m_{2},m_{3};s_{1},s_{2};z\right)  \\
+\left( B_{2}+C_{2}z\right) \ _{3}F_{2}\left(
m_{1},m_{2},m_{3};s_{1}+1,s_{2};z\right) +C_{3}z\ _{3}F_{2}\left(
m_{1},m_{2},m_{3};s_{1}+2,s_{2};z\right)
\end{multline}
where%
\begin{gather*}
B_{1} =\frac{s_{2}-2s_{1}}{s_{1}-1},\quad B_{2}=\frac{s_{1}-s_{2}+1}{s_{1}-1}%
,\quad C_{1}=\frac{3s_{1}-m_{1}-m_{2}-m_{3}}{s_{1}-1} , \\
C_{2} =\frac{\left( 2s_{1}+1\right) \left( m_{1}+m_{2}+m_{3}\right)
-3\left( s_{1}-1\right) \left( s_{1}+2\right)
-7-m_{1}m_{2}-m_{1}m_{3}-m_{2}m_{3}}{\left( s_{1}-1\right) _{2}}
\end{gather*}%
and%
\begin{equation*}
C_{3}=\frac{\left( s_{1}-m_{1}+1\right) \left( s_{1}-m_{2}+1\right) \left(
s_{1}-m_{3}+1\right) }{\left( s_{1}-1\right) _{3}},
\end{equation*}

\emph{(iii)}%
\begin{multline}\label{H3}
s_{1}\ _{3}F_{2}\left( m_{1},m_{2},m_{3};s_{1},s_{2};z\right) +\left(
m_{1}-s_{1}\right) \ _{3}F_{2}\left( m_{1},m_{2},m_{3};s_{1}+1,s_{2};z\right)\\
=m_{1}\ _{3}F_{2}\left( m_{1}+1,m_{2},m_{3};s_{1}+1,s_{2};z\right) .
\end{multline}
\end{lemma}

From the relation
\begin{multline*}
S_{n}\left( x;a_{1},a_{2},b_{1},b_{2}\right)
:=  _{1}S_{n_{1}}\left( x_{1};a_{1},a_{2},b_{1},b_{2}\right) \\
=\hyper{3}{2}{-n_{1} ,\
a_{2}+x_{1} /2,\ n_{1} +a_{1}+a_{2}+b_{1}+b_{2}-1}{a_{2}+b_{2},\ a_{1}+a_{2}}{1}
\end{multline*}
given by \eqref{1boyut}, we first give the following lemmas to derive recurrence relations for the functions $_{r}S_{\mathbf{n}}\left( \mathbf{x};\mathbf{a},\mathbf{b}\right) $ for $r>1$.

\begin{lemma}
The family of the special function$\ _{1}S_{n_{1}}\left(
x_{1};a_{1},a_{2},b_{1},b_{2}\right) $ satisfies the recurrence relation%
\begin{eqnarray}
\left( B_{1}+C_{1}\right) \ _{1}S_{n_{1}}\left(
x_{1};a_{1},a_{2},b_{1},b_{2}\right) +\left( B_{2}+C_{2}\right) \
_{1}S_{n_{1}+1}\left( x_{1};a_{1},a_{2},b_{1}-1,b_{2}\right)  &&  \label{*1}
\\
+B_{3}\ _{1}S_{n_{1}+2}\left( x_{1};a_{1},a_{2},b_{1}-2,b_{2}\right) =0, &&
\notag
\end{eqnarray}%
where
\begin{eqnarray*}
B_{1}+C_{1} &=&\frac{b_{1}-1+\frac{x_{1}}{2}}{n_{1}}~~,~~B_{3}=-\frac{\left(
n_{1}+a_{2}+b_{2}+1\right) \left( n_{1}+a_{1}+a_{2}+1\right) }{\left(
n_{1}\right) _{2}}, \\
B_{2}+C_{2} &=&\frac{\left( a_{2}+b_{2}\right) \left( a_{1}+a_{2}\right)
-\left( n_{1}+a_{1}+a_{2}+b_{1}+b_{2}-1\right) \left( a_{2}+\frac{x_{1}}{2}%
\right) }{\left( n_{1}\right) _{2}} \\
&&+\frac{n_{1}+a_{1}+2a_{2}-b_{1}+b_{2}+2-\frac{x_{1}}{2}}{n_{1}}.
\end{eqnarray*}
\end{lemma}

\begin{proof}
Substituting $m_{1}\rightarrow -n_{1}$, $m_{2}\rightarrow
n_{1}+a_{1}+a_{2}+b_{1}+b_{2}-1$, $m_{3}\rightarrow a_{2}+\frac{x_{1}}{2}$, $%
s_{1}\rightarrow a_{2}+b_{2}$, $s_{2}\rightarrow a_{1}+a_{2}$ and $%
z\rightarrow 1$ in the relation \eqref{H1}, we complete the proof.
\end{proof}

\begin{lemma}
The family of the special function$\ _{1}S_{n_{1}}\left(
x_{1};a_{1},a_{2},b_{1},b_{2}\right) $ satisfies the recurrence relation%
\begin{eqnarray}
\left( B_{1}+C_{1}\right) \ _{1}S_{n_{1}}\left(
x_{1};a_{1},a_{2},b_{1},b_{2}\right) +\left( B_{2}+C_{2}\right) \
_{1}S_{n_{1}}\left( x_{1};a_{1},a_{2},b_{1}-1,b_{2}\right)  &&  \label{*2} \\
+B_{3}\ _{1}S_{n_{1}}\left( x_{1};a_{1},a_{2},b_{1}-2,b_{2}\right) =0, &&
\notag
\end{eqnarray}%
where
\begin{eqnarray*}
B_{1}+C_{1} &=&-\frac{b_{1}-1+\frac{x_{1}}{2}}{%
n_{1}+a_{1}+a_{2}+b_{1}+b_{2}-1}, \\
B_{2}+C_{2} &=&\frac{\left( a_{2}+b_{2}\right) \left( a_{1}+a_{2}\right)
+n_{1}\left( a_{2}+\frac{x_{1}}{2}\right) }{\left(
n_{1}+a_{1}+a_{2}+b_{1}+b_{2}-2\right) _{2}}+\frac{n_{1}-a_{2}+2b_{1}-3+%
\frac{x_{1}}{2}}{n_{1}+a_{1}+a_{2}+b_{1}+b_{2}-1}, \\
B_{3} &=&-\frac{\left( n_{1}+a_{1}+b_{1}-2\right) \left(
n_{1}+b_{1}+b_{2}-2\right) }{\left( n_{1}+a_{1}+a_{2}+b_{1}+b_{2}-2\right)
_{2}}.
\end{eqnarray*}
\end{lemma}

\begin{proof}
By taking $m_{1}\rightarrow n_{1}+a_{1}+a_{2}+b_{1}+b_{2}-1$, $%
m_{2}\rightarrow -n_{1}$, $m_{3}\rightarrow a_{2}+\frac{x_{1}}{2}$, $%
s_{1}\rightarrow a_{2}+b_{2}$, $s_{2}\rightarrow a_{1}+a_{2}$ and $%
z\rightarrow 1$ in the relation (\ref{H1}), we obtain the desired relation.
\end{proof}

\begin{lemma}
For the family of the special function$\ _{1}S_{n_{1}}\left(
x_{1};a_{1},a_{2},b_{1},b_{2}\right) $, the following relation holds%
\begin{multline}\label{*3}
\left( B_{1}+C_{1}\right) \ _{1}S_{n_{1}}\left(
x_{1};a_{1},a_{2},b_{1},b_{2}\right) +\left( B_{2}+C_{2}\right) \
_{1}S_{n_{1}}\left( x_{1};a_{1}+1,a_{2}-1,b_{1}-1,b_{2}+1\right) \\
+B_{3}\ _{1}S_{n_{1}}\left( x_{1};a_{1}+2,a_{2}-2,b_{1}-2,b_{2}+2\right) =0,
\end{multline}%
where
\begin{eqnarray*}
B_{1}+C_{1} &=&-\frac{b_{1}-1+\frac{x_{1}}{2}}{a_{2}+\frac{x_{1}}{2}}%
~~,~~B_{3}=-\frac{\left( a_{1}+1-\frac{x_{1}}{2}\right) \left( b_{2}+1-\frac{%
x_{1}}{2}\right) }{\left( a_{2}-1+\frac{x_{1}}{2}\right) _{2}}, \\
B_{2}+C_{2} &=&\frac{\left( a_{2}+b_{2}\right) \left( a_{1}+a_{2}\right)
+n_{1}\left( n_{1}+a_{1}+a_{2}+b_{1}+b_{2}-1\right) }{\left( a_{2}-1+\frac{%
x_{1}}{2}\right) _{2}} \\
&&-\frac{a_{1}+a_{2}-b_{1}+b_{2}+2-x_{1}}{a_{2}+\frac{x_{1}}{2}}.
\end{eqnarray*}
\end{lemma}

\begin{proof}
If we get $m_{1}\rightarrow a_{2}+\frac{x_{1}}{2}$, $m_{2}\rightarrow
-n_{1}$, $m_{3}\rightarrow n_{1}+a_{1}+a_{2}+b_{1}+b_{2}-1$, $%
s_{1}\rightarrow a_{2}+b_{2}$, $s_{2}\rightarrow a_{1}+a_{2}$ and $%
z\rightarrow 1$ in relation \eqref{H1}, the above identity is arrived.
\end{proof}

\begin{lemma}
The family of the special function$\ _{1}S_{n_{1}}\left(
x_{1};a_{1},a_{2},b_{1},b_{2}\right) $ satisfies the relation%
\begin{multline} \label{*4}
\left( B_{1}+C_{1}\right) \ _{1}S_{n_{1}}\left(
x_{1};a_{1},a_{2},b_{1},b_{2}\right) +\left( B_{2}+C_{2}\right) \
_{1}S_{n_{1}}\left( x_{1};a_{1}+1,a_{2},b_{1}-1,b_{2}\right)  
\\
+C_{3}\ _{1}S_{n_{1}}\left( x_{1};a_{1}+2,a_{2},b_{1}-2,b_{2}\right) =0,
\end{multline}%
where
\begin{align*}
B_{1}+C_{1} &=-\frac{b_{1}-1+\frac{x_{1}}{2}}{a_{1}+a_{2}-1}, \\
B_{2}+C_{2} &=\frac{\left( 2\left( a_{1}+a_{2}\right) +1\right) \left(
a_{2}+b_{1}+b_{2}-2+\frac{x_{1}}{2}\right) +n_{1}\left(
n_{1}+a_{1}+a_{2}+b_{1}+b_{2}-1\right) }{\left( a_{1}+a_{2}-1\right) _{2}} \\
&-\frac{\left( a_{1}+a_{2}+b_{1}+b_{2}-1\right) \left( a_{2}+\frac{x_{1}}{2}%
\right) }{\left( a_{1}+a_{2}-1\right) _{2}}-\frac{a_{2}+b_{2}-1}{%
a_{1}+a_{2}-1}, \\
C_{3} &=-\frac{\left( n_{1}+a_{1}+a_{2}+1\right) \left(
n_{1}+b_{1}+b_{2}-2\right) \left( a_{1}+1-\frac{x_{1}}{2}\right) }{\left(
a_{1}+a_{2}-1\right) _{3}}.
\end{align*}
\end{lemma}

\begin{proof}
Taking $m_{1}\rightarrow -n_{1}$, $m_{2}\rightarrow
n_{1}+a_{1}+a_{2}+b_{1}+b_{2}-1$, $m_{3}\rightarrow a_{2}+\frac{x_{1}}{2}$, $%
s_{1}\rightarrow a_{1}+a_{2}$, $s_{2}\rightarrow a_{2}+b_{2}$ and $%
z\rightarrow 1$ in \eqref{H2}, the proof is completed.
\end{proof}

\begin{lemma}
The family of the special function$\ _{1}S_{n_{1}}\left(
x_{1};a_{1},a_{2},b_{1},b_{2}\right) $ satisfies the recurrence relation%
\begin{multline}\label{*6}
\left( B_{1}+C_{1}\right) \ _{1}S_{n_{1}}\left(
x_{1};a_{1},a_{2},b_{1},b_{2}\right) +\left( B_{2}+C_{2}\right) \
_{1}S_{n_{1}}\left( x_{1};a_{1},a_{2},b_{1}-1,b_{2}+1\right)  \\
+C_{3}\ _{1}S_{n_{1}}\left( x_{1};a_{1},a_{2},b_{1}-2,b_{2}+2\right) =0,
\end{multline}%
where
\begin{align*}
B_{1}+C_{1} &=-\frac{b_{1}-1+\frac{x_{1}}{2}}{a_{2}+b_{2}-1}, \\
B_{2}+C_{2} &=\frac{\left( 2\left( a_{2}+b_{2}\right) +1\right) \left(
a_{1}+a_{2}+b_{1}-2+\frac{x_{1}}{2}\right) +n_{1}\left(
n_{1}+a_{1}+a_{2}+b_{1}+b_{2}-1\right) }{\left( a_{2}+b_{2}-1\right) _{2}} \\
&-\frac{\left( a_{1}+a_{2}+b_{1}+b_{2}-1\right) \left( a_{2}+\frac{x_{1}}{2}%
\right) }{\left( a_{2}+b_{2}-1\right) _{2}}-\frac{a_{1}+a_{2}-1}{%
a_{2}+b_{2}-1}, \\
C_{3} &=-\frac{\left( n_{1}+a_{2}+b_{2}+1\right) \left(
n_{1}+a_{1}+b_{1}-2\right) \left( b_{2}+1-\frac{x_{1}}{2}\right) }{\left(
a_{2}+b_{2}-1\right) _{3}}.
\end{align*}
\end{lemma}

\begin{proof}
Replacing $m_{1}\rightarrow -n_{1}$, $m_{2}\rightarrow
n_{1}+a_{1}+a_{2}+b_{1}+b_{2}-1$, $m_{3}\rightarrow a_{2}+\frac{x_{1}}{2}$, $%
s_{1}\rightarrow a_{2}+b_{2}$, $s_{2}\rightarrow a_{1}+a_{2}$ and $%
z\rightarrow 1$ in relation \eqref{H2}, the relation \eqref{*6} is obtained.
\end{proof}

\begin{lemma}
The family of the special function$\ _{1}S_{n_{1}}\left(
x_{1};a_{1},a_{2},b_{1},b_{2}\right) $ verifies the recurrence relation%
\begin{multline}\label{*7}
\left( a_{1}+a_{2}\right) \ _{1}S_{n_{1}}\left(
x_{1};a_{1},a_{2},b_{1},b_{2}\right) +\left( n_{1}+b_{1}+b_{2}-1\right) \
_{1}S_{n_{1}}\left( x_{1};a_{1}+1,a_{2},b_{1}-1,b_{2}\right)
\\
-\left( n_{1}+a_{1}+a_{2}+b_{1}+b_{2}-1\right) \ _{1}S_{n_{1}}\left(
x_{1};a_{1}+1,a_{2},b_{1},b_{2}\right) =0.
\end{multline}
\end{lemma}

\begin{proof}
By taking $m_{1}\rightarrow n_{1}+a_{1}+a_{2}+b_{1}+b_{2}-1$, $m_{2}\rightarrow
-n_{1}$, $m_{3}\rightarrow a_{2}+\frac{x_{1}}{2}$, $s_{1}\rightarrow
a_{1}+a_{2}$, $s_{2}\rightarrow a_{2}+b_{2}$ and $z\rightarrow 1$ in
relation \eqref{H3}, it is verified.
\end{proof}

From \eqref{Func1} it is easily seen that the function $_{2}S_{n_{1},n_{2}}\left(
x_{1},x_{2};a_{1},a_{2},a_{3},b_{1},b_{2},b_{3}\right)$ is written in terms of $_{1}S_{n_{1}}\left(
x_{1};a_{1},a_{2},b_{1},b_{2}\right)$ as follows
\begin{multline}\label{relation1}
_{2}S_{n_{1},n_{2}}\left(
x_{1},x_{2};a_{1},a_{2},a_{3},b_{1},b_{2},b_{3}\right)
=~_{1}S_{n_{1}}\left(
x_{1};a_{1},a_{2}+a_{3}+n_{2},b_{1},b_{2}+b_{3}+n_{2}\right)  \\
\times \hyper{3}{2}{-n_{2} ,\
a_{3}+x_{2} /2,\ n_{2} +a_{2}+a_{3}+b_{2}+b_{3}-1}{a_{3}+b_{3},\ a_{2}+a_{3}}{1}  \left( a_{2}+a_{3}+\frac{x_{1}}{2}\right) _{n_{2}}
\end{multline}%
and
\begin{multline}\label{relation2}
_{2}S_{n_{1},n_{2}}\left(
x_{1},x_{2};a_{1},a_{2},a_{3},b_{1},b_{2},b_{3}\right)
= \left( a_{2}+a_{3}+\frac{x_{1}}{2}\right)_{n_{2}}\,  _{1}S_{n_{2}}\left(
x_{2};a_{2},a_{3},b_{2},b_{3}\right)   \\
\times \hyper{3}{2}{-n_{1} ,\
n_{2}+a_{2}+a_{3}+x_{1} /2,\ n_{1}+2n_{2} +a_{1}+a_{2}+a_{3}+b_{1}+b_{2}+b_{3}-1}{2n_{2}+a_{2}+a_{3}+b_{2}+b_{3},\ n_{2}+a_{1}+a_{2}+a_{3}}{1}.
\end{multline}

By using these relations and lemmas given above for $_{1}S_{n_{1}}\left(
x_{1};a_{1},a_{2},b_{1},b_{2}\right) ,$ we can derive several recurrence
relations for the function $_{2}S_{n_{1},n_{2}}\left(
x_{1},x_{2};a_{1},a_{2},a_{3},b_{1},b_{2},b_{3}\right) .$

\begin{theorem}
The family of the special function$\ _{2}S_{n_{1},n_{2}}\left(
x_{1},x_{2};a_{1},a_{2},a_{3},b_{1},b_{2},b_{3}\right) $ satisfies the recurrence relation%
\begin{multline}\label{101}
\left( B_{1}+C_{1}\right) \ _{2}S_{n_{1},n_{2}}\left(
x_{1},x_{2};a_{1},a_{2},a_{3},b_{1},b_{2},b_{3}\right)    \\
+\left( B_{2}+C_{2}\right) \ _{2}S_{n_{1}+1,n_{2}}\left(
x_{1},x_{2};a_{1},a_{2},a_{3},b_{1}-1,b_{2},b_{3}\right)    \\
+B_{3}\ _{2}S_{n_{1}+2,n_{2}}\left(
x_{1},x_{2};a_{1},a_{2},a_{3},b_{1}-2,b_{2},b_{3}\right) =0,
\end{multline}%
where
\begin{align*}
B_{1}+C_{1} &=\frac{b_{1}-1+\frac{x_{1}}{2}}{n_{1}}, \\
B_{2}+C_{2} &=\frac{\left( 2n_{2}+a_{2}+a_{3}+b_{2}+b_{3}\right) \left(
n_{2}+a_{1}+a_{2}+a_{3}\right) }{\left( n_{1}\right) _{2}} \\
&-\frac{\left( n_{1}+2n_{2}+a_{1}+a_{2}+a_{3}+b_{1}+b_{2}+b_{3}-1\right)
\left( n_{2}+a_{2}+a_{3}+\frac{x_{1}}{2}\right) }{\left( n_{1}\right) _{2}}
\\
&+\frac{n_{1}+3n_{2}+a_{1}+2\left( a_{2}+a_{3}\right) -b_{1}+b_{2}+b_{3}+2-%
\frac{x_{1}}{2}}{n_{1}}, \\
B_{3} &=-\frac{\left( n_{1}+2n_{2}+a_{2}+a_{3}+b_{2}+b_{3}+1\right) \left(
n_{1}+n_{2}+a_{1}+a_{2}+a_{3}+1\right) }{\left( n_{1}\right) _{2}}.
\end{align*}
\end{theorem}

\begin{proof}
Substituting $a_{2}\rightarrow n_{2}+a_{2}+a_{3}$ and $b_{2}\rightarrow
n_{2}+b_{2}+b_{3}$ in the relation \eqref{*1}, if the obtained identity is
multiplied by
$$
\left( a_{2}+a_{3}+\frac{x_{1}}{2}\right) _{n_{2}} \hyper{3}{2}{-n_{2} ,\
a_{3}+x_{2} /2,\ n_{2} +a_{2}+a_{3}+b_{2}+b_{3}-1}{a_{3}+b_{3},\ a_{2}+a_{3}}{1},
$$
then the proof is completed from \eqref{relation1}.
\end{proof}

\begin{theorem}
For the family of the special function$\ _{2}S_{n_{1},n_{2}}\left(
x_{1},x_{2};a_{1},a_{2},a_{3},b_{1},b_{2},b_{3}\right) $, we have
\begin{multline}\label{102}
\left( B_{1}+C_{1}\right) \ _{2}S_{n_{1},n_{2}}\left(
x_{1},x_{2};a_{1},a_{2},a_{3},b_{1},b_{2},b_{3}\right)    \\
+\left( B_{2}+C_{2}\right) \ _{2}S_{n_{1},n_{2}}\left(
x_{1},x_{2};a_{1},a_{2},a_{3},b_{1}-1,b_{2},b_{3}\right)    \\
+B_{3}\ _{2}S_{n_{1},n_{2}}\left(
x_{1},x_{2};a_{1},a_{2},a_{3},b_{1}-2,b_{2},b_{3}\right) =0,
\end{multline}%
where
\begin{align*}
B_{1}+C_{1} &=-\frac{b_{1}-1+\frac{x_{1}}{2}}{%
n_{1}+2n_{2}+a_{1}+a_{2}+a_{3}+b_{1}+b_{2}+b_{3}-1}, \\
B_{2}+C_{2} &=\frac{\left( 2n_{2}+a_{2}+a_{3}+b_{2}+b_{3}\right) \left(
n_{2}+a_{1}+a_{2}+a_{3}\right) +n_{1}\left( n_{2}+a_{2}+a_{3}+\frac{x_{1}}{2}%
\right) }{\left( n_{1}+2n_{2}+a_{1}+a_{2}+a_{3}+b_{1}+b_{2}+b_{3}-2\right)
_{2}} \\
&+\frac{n_{1}-n_{2}-\left( a_{2}+a_{3}\right) +2b_{1}-3+\frac{x_{1}}{2}}{%
n_{1}+2n_{2}+a_{1}+a_{2}+a_{3}+b_{1}+b_{2}+b_{3}-1}, \\
B_{3} &=-\frac{\left( n_{1}+a_{1}+b_{1}-2\right) \left(
n_{1}+n_{2}+b_{1}+b_{2}+b_{3}-2\right) }{\left(
n_{1}+2n_{2}+a_{1}+a_{2}+a_{3}+b_{1}+b_{2}+b_{3}-2\right) _{2}}.
\end{align*}
\end{theorem}

\begin{proof}
After $a_{2}\rightarrow n_{2}+a_{2}+a_{3}$ and $b_{2}\rightarrow
n_{2}+b_{2}+b_{3}$ in relation \eqref{*2}, if we multiply
multiplied by
$$
\left( a_{2}+a_{3}+\frac{x_{1}}{2}\right) _{n_{2}} \hyper{3}{2}{-n_{2} ,\
a_{3}+x_{2} /2,\ n_{2} +a_{2}+a_{3}+b_{2}+b_{3}-1}{a_{3}+b_{3},\ a_{2}+a_{3}}{1},
$$
we arrive at the desired relation from \eqref{relation1}.
\end{proof}

\begin{theorem}
The family of the special function$\ _{2}S_{n_{1},n_{2}}\left(
x_{1},x_{2};a_{1},a_{2},a_{3},b_{1},b_{2},b_{3}\right) $ satisfies the
relation%
\begin{multline}\label{103}
\left( B_{1}+C_{1}\right) \ _{2}S_{n_{1},n_{2}}\left(
x_{1},x_{2};a_{1},a_{2},a_{3},b_{1},b_{2},b_{3}\right)    \\
+\left( B_{2}+C_{2}\right) \ _{2}S_{n_{1},n_{2}}\left(
x_{1},x_{2};a_{1},a_{2}+1,a_{3}-1,b_{1},b_{2}-1,b_{3}+1\right)    \\
+B_{3}\ _{2}S_{n_{1},n_{2}}\left(
x_{1},x_{2};a_{1},a_{2}+2,a_{3}-2,b_{1},b_{2}-2,b_{3}+2\right) =0,
\end{multline}%
where
\begin{align*}
B_{1}+C_{1} &=-\frac{b_{2}-1+\frac{x_{2}}{2}}{a_{3}+\frac{x_{2}}{2}}, \\
B_{2}+C_{2} &=\frac{\left( a_{2}+a_{3}\right) \left( a_{3}+b_{3}\right)
+n_{2}\left( n_{2}+a_{2}+a_{3}+b_{2}+b_{3}-1\right) }{\left( a_{3}-1+\frac{%
x_{2}}{2}\right) _{2}} \\
&-\frac{a_{2}+a_{3}-b_{2}+b_{3}+2-x_{2}}{a_{3}+\frac{x_{2}}{2}}, \\
B_{3} &=\frac{\left( a_{2}+1-\frac{x_{2}}{2}\right) \left( b_{3}+1-\frac{%
x_{2}}{2}\right) }{\left( a_{3}-1+\frac{x_{2}}{2}\right) _{2}}.
\end{align*}
\end{theorem}

\begin{proof}
After substituting $n_{1}\rightarrow n_{2}$, $x_{1}\rightarrow x_{2}$, $%
a_{1}\rightarrow a_{2}$, $a_{2}\rightarrow a_{3}$, $b_{1}\rightarrow b_{2}$
and $b_{2}\rightarrow b_{3}$ in relation \eqref{*3}, and then multiplying by
\begin{multline*}
\left( a_{2}+a_{3}+\frac{x_{1}}{2}\right) _{n_{2}} \\
\times \hyper{3}{2}{-n_{1} ,\
n_{2}+a_{2}+a_{3}+x_{1} /2,\ n_{1}+2n_{2} +a_{1}+a_{2}+a_{3}+b_{1}+b_{2}+b_{3}-1}{2n_{2}+a_{2}+a_{3}+b_{2}+b_{3},\ n_{2}+a_{1}+a_{2}+a_{3}}{1},
\end{multline*}
it follows from \eqref{relation2}.
\end{proof}

\begin{theorem}
For the family of the special function$\ _{2}S_{n_{1},n_{2}}\left(
x_{1},x_{2};a_{1},a_{2},a_{3},b_{1},b_{2},b_{3}\right) $, we have
\begin{multline}\label{104}
\left( B_{1}+C_{1}\right) \ _{2}S_{n_{1},n_{2}}\left(
x_{1},x_{2};a_{1},a_{2},a_{3},b_{1},b_{2},b_{3}\right)    \\
+\left( B_{2}+C_{2}\right) \ _{2}S_{n_{1},n_{2}}\left(
x_{1},x_{2};a_{1}+1,a_{2},a_{3},b_{1}-1,b_{2},b_{3}\right)    \\
+C_{3}\ _{2}S_{n_{1},n_{2}}\left(
x_{1},x_{2};a_{1}+2,a_{2},a_{3},b_{1}-2,b_{2},b_{3}\right) =0,
\end{multline}%
where
\begin{align*}
B_{1}+C_{1} &=-\frac{b_{1}-1+\frac{x_{1}}{2}}{n_{2}+a_{1}+a_{2}+a_{3}-1} \\
C_{3} &=-\frac{\left( n_{1}+n_{2}+a_{1}+a_{2}+a_{3}+1\right) \left(
n_{1}+n_{2}+b_{1}+b_{2}+b_{3}-2\right) \left( a_{1}+1-\frac{x_{1}}{2}\right)
}{\left( n_{2}+a_{1}+a_{2}+a_{3}-1\right) _{3}}, \\
B_{2}+C_{2} &=\frac{\left( 2n_{2}+2\left( a_{1}+a_{2}+a_{3}\right)
+1\right) \left( 2n_{2}+a_{2}+a_{3}+b_{1}+b_{2}+b_{3}-2+\frac{x_{1}}{2}%
\right) }{\left( n_{2}+a_{1}+a_{2}+a_{3}-1\right) _{2}} \\
&+\frac{n_{1}\left(
n_{1}+2n_{2}+a_{1}+a_{2}+a_{3}+b_{1}+b_{2}+b_{3}-1\right) }{\left(
n_{2}+a_{1}+a_{2}+a_{3}-1\right) _{2}} \\
&-\frac{\left( 2n_{2}+a_{1}+a_{2}+a_{3}+b_{1}+b_{2}+b_{3}-1\right) \left(
n_{2}+a_{2}+a_{3}+\frac{x_{1}}{2}\right) }{\left(
n_{2}+a_{1}+a_{2}+a_{3}-1\right) _{2}} \\
&-\frac{2n_{2}+a_{2}+a_{3}+b_{2}+b_{3}-1}{n_{2}+a_{1}+a_{2}+a_{3}-1}.
\end{align*}
\end{theorem}

\begin{proof}
If we get $a_{2}\rightarrow n_{2}+a_{2}+a_{3}$ and $b_{2}\rightarrow
n_{2}+b_{2}+b_{3}$ in \eqref{*4} and then we multiply it by
$$
\left( a_{2}+a_{3}+\frac{x_{1}}{2}\right) _{n_{2}} \hyper{3}{2}{-n_{2} ,\
a_{3}+x_{2} /2,\ n_{2} +a_{2}+a_{3}+b_{2}+b_{3}-1}{a_{3}+b_{3},\ a_{2}+a_{3}}{1},
$$
we complete the proof from \eqref{relation1}.
\end{proof}

\begin{theorem}
The family of the special function$\ _{2}S_{n_{1},n_{2}}\left(
x_{1},x_{2};a_{1},a_{2},a_{3},b_{1},b_{2},b_{3}\right) $ verifies the
recurrence relation%
\begin{multline}\label{106}
\left( B_{1}+C_{1}\right) \ _{2}S_{n_{1},n_{2}}\left(
x_{1},x_{2};a_{1},a_{2},a_{3},b_{1},b_{2},b_{3}\right)    \\
+\left( B_{2}+C_{2}\right) \ _{2}S_{n_{1},n_{2}}\left(
x_{1},x_{2};a_{1},a_{2},a_{3},b_{1},b_{2}-1,b_{3}+1\right)    \\
+C_{3}\ _{2}S_{n_{1},n_{2}}\left(
x_{1},x_{2};a_{1},a_{2},a_{3},b_{1},b_{2}-2,b_{3}+2\right) =0,
\end{multline}%
where
\begin{align*}
B_{1}+C_{1} &=-\frac{b_{2}-1+\frac{x_{2}}{2}}{a_{3}+b_{3}-1}, \\
B_{2}+C_{2} &=\frac{\left( 2\left( a_{3}+b_{3}\right) +1\right) \left(
a_{2}+a_{3}+b_{2}-2+\frac{x_{2}}{2}\right) }{\left( a_{3}+b_{3}-1\right) _{2}%
}-\frac{a_{2}+a_{3}-1}{a_{3}+b_{3}-1} \\
&+\frac{n_{2}\left( n_{2}+a_{2}+a_{3}+b_{2}+b_{3}-1\right) -\left(
a_{2}+a_{3}+b_{2}+b_{3}-1\right) \left( a_{3}+\frac{x_{2}}{2}\right) }{%
\left( a_{3}+b_{3}-1\right) _{2}}, \\
C_{3} &=-\frac{\left( n_{2}+a_{2}+b_{2}-2\right) \left(
n_{2}+a_{3}+b_{3}+1\right) \left( b_{3}+1-\frac{x_{2}}{2}\right) }{\left(
a_{3}+b_{3}-1\right) _{3}}.
\end{align*}
\end{theorem}

\begin{proof}
After replacing $n_{1}\rightarrow n_{2}$, $x_{1}\rightarrow x_{2}$, $%
a_{1}\rightarrow a_{2}$, $a_{2}\rightarrow a_{3}$, $b_{1}\rightarrow b_{2}$
and $b_{2}\rightarrow b_{3}$ in relation \eqref{*6}, if the obtained identity
is multiplied by
\begin{multline*}
\left( a_{2}+a_{3}+\frac{x_{1}}{2}\right) _{n_{2}}\\
\times  \hyper{3}{2}{-n_{1} ,\
n_{2}+a_{2}+a_{3}+x_{1} /2,\ n_{1}+2n_{2} +a_{1}+a_{2}+a_{3}+b_{1}+b_{2}+b_{3}-1}{2n_{2}+a_{2}+a_{3}+b_{2}+b_{3},\ n_{2}+a_{1}+a_{2}+a_{3}}{1},
\end{multline*}
the proof is completed from \eqref{relation2}.
\end{proof}

\begin{theorem}
The family of the special function$\ _{2}S_{n_{1},n_{2}}\left(
x_{1},x_{2};a_{1},a_{2},a_{3},b_{1},b_{2},b_{3}\right) $ satisfies the
recurrence relation%
\begin{multline}\label{107}
\left( n_{2}+a_{1}+a_{2}+a_{3}\right) \ _{2}S_{n_{1},n_{2}}\left(
x_{1},x_{2};a_{1},a_{2},a_{3},b_{1},b_{2},b_{3}\right)   \\
-\left( n_{1}+2n_{2}+a_{1}+a_{2}+a_{3}+b_{1}+b_{2}+b_{3}-1\right) \
_{2}S_{n_{1},n_{2}}\left(
x_{1},x_{2};a_{1}+1,a_{2},a_{3},b_{1},b_{2},b_{3}\right)   \\
+\left( n_{1}+n_{2}+b_{1}+b_{2}+b_{3}-1\right) \ _{2}S_{n_{1},n_{2}}\left(
x_{1},x_{2};a_{1}+1,a_{2},a_{3},b_{1}-1,b_{2},b_{3}\right) =0.
\end{multline}
\end{theorem}

\begin{proof}
Taking $a_{2}\rightarrow n_{2}+a_{2}+a_{3}$ and $b_{2}\rightarrow
n_{2}+b_{2}+b_{3}$ in relation \eqref{*7},  in view of the relation \eqref{relation1} we arrive at the desired relation.
\end{proof}

We now derive recurrence relations for the special functions $
_{r}S_{\mathbf{n}}\left( \mathbf{x};\mathbf{a},\mathbf{b}\right) $.
By taking into account \eqref{Func1}, we first give the following relationships between the special functions$\
_{r}S_{\mathbf{n}}\left( \mathbf{x};\mathbf{a},\mathbf{b}\right) $ as
follows

\begin{multline} \label{P1}
 _{r}S_{\mathbf{n}}\left( \mathbf{x};\mathbf{a},\mathbf{b}\right)\\
 = \hyper{3}{2}{-n_{r} ,\ n_{r}+\left \vert \mathbf{a}^{r}\right \vert +\left \vert \mathbf{b}^{r}\right \vert -1,\ a_{r+1}+
 \frac{x_{r}}{2}}{a_{r+1}+b_{r+1},\left \vert \mathbf{a}^{r}\right \vert }
{1} \prod \limits_{j=1}^{r-1}\left( \left \vert \mathbf{a}
^{j+1}\right \vert +\frac{x_{j}}{2}\right) _{n_{r}}\\
\times _{r-1}S_{n_{1},\dots ,n_{r-1}}\left(
x_{1},\dots ,x_{r-1};a_{1},\dots ,a_{r-1},a_{r}+a_{r+1}+n_{r},b_{1},\dots ,b_{r-1},b_{r}+b_{r+1}+n_{r}\right)
\end{multline}
or%
\begin{multline}\label{P2}
_{r}S_{\mathbf{n}}\left( \mathbf{x};\mathbf{a},\mathbf{b}\right)  =\left(
\left \vert \mathbf{a}^{2}\right \vert +\frac{x_{1}}{2}\right) _{\left \vert
\mathbf{n}^{2}\right \vert }~_{r-1}S_{n_{2},\dots ,n_{r}}\left(
x_{2},\dots ,x_{r}; \mathbf{a}^{2} ,\mathbf{b}%
^{2} \right)    \\
\times \hyper{3}{2}{-n_{1} ,\ n_{1}+2\left \vert \mathbf{n}%
^{2}\right \vert +\left \vert \mathbf{a}\right \vert +\left \vert \mathbf{b}%
\right \vert -1,\left \vert \mathbf{n}^{2}\right \vert +\left \vert \mathbf{a}%
^{2}\right \vert +\frac{x_{1}}{2}}{2\left \vert \mathbf{n}^{2}\right \vert
+\left \vert \mathbf{a}^{2}\right \vert +\left \vert \mathbf{b}^{2}\right \vert,\left \vert \mathbf{n}^{2}\right \vert +\left \vert \mathbf{a}\right \vert }{1} ,
\end{multline}%
for $r\geq 2$, where $\boldsymbol{n}:=\left(  n_{1},n_{2},\dots,n_{r}\right)  $,
 $\mathbf{a}:=\left(  a_{1},a_{2},\dots,a_{r+1}\right)  $ and $\mathbf{b}:=\left(  b_{1},b_{2},\dots,b_{r+1}\right)  $ also
\begin{align*}
\mathbf{a}^{j}  & =\left(  a_{j},\dots,a_{r+1}\right)  ,\  \  \  \ 1\leq j\leq r+1,\\
\mathbf{b}^{j}  & =\left(  b_{j},\dots,b_{r+1}\right)   ,\  \  \  \ 1\leq j\leq r+1 \\
\mathbf{n}^{j}  & =\left(  n_{j},\dots,n_{r}\right)  ,\  \  \  \ 1\leq j\leq r
\end{align*}
so $\left \vert \mathbf{a}^{j}\right \vert =a_{j}+a_{j+1}+\cdots+a_{r+1}$, $\left \vert \mathbf{b}^{j}\right \vert =b_{j}+b_{j+1}+\cdots+b_{r+1}$ and $\left \vert \mathbf{n}^{j}\right \vert =n_{j}+n_{j+1}+\dots+n_{r}$.\\

In order to derive the recurrence relations for the function $_{3}S_{n_{1},n_{2},n_{3}}$ depending on $\left(
x_{1},x_{2},x_{3};a_{1},a_{2},a_{3},a_{4},b_{1},b_{2},b_{3},b_{4}\right)$, it is enough to use Theorems 5.8-5.13 given for the function $_{2}S_{n_{1},n_{2}}\left(
x_{1},x_{2};a_{1},a_{2},a_{3},b_{1},b_{2},b_{3}\right)$ by means of the relationships \eqref{P1} and \eqref{P2}. When similar process is repeated consecutively, the following results appear.

\begin{theorem}
For $r\geq 1$, the family of the special function$\ _{r}S_{\mathbf{n}}\left(
\mathbf{x};\mathbf{a},\mathbf{b}\right) $ satisfies the following recurrence
relation%
\begin{multline*}
\left( B_{1}+C_{1}\right) \ _{r}S_{\mathbf{n}}\left( \mathbf{x};\mathbf{a},%
\mathbf{b}\right) +\left( B_{2}+C_{2}\right) \
_{r}S_{n_{1}+1,n_{2},\dots ,n_{r}}\left( \mathbf{x};\mathbf{a}%
,b_{1}-1,b_{2},\dots ,b_{r+1}\right)  \\
+B_{3}\ _{r}S_{n_{1}+2,n_{2},\dots ,n_{r}}\left( \mathbf{x};\mathbf{a}%
,b_{1}-2,b_{2},\dots ,b_{r+1}\right) =0,
\end{multline*}%
where
\begin{eqnarray*}
B_{1}+C_{1} &=&\frac{b_{1}-1+\frac{x_{1}}{2}}{n_{1}}, \\
B_{2}+C_{2} &=&\frac{\left( 2\left \vert \mathbf{n}^{2}\right \vert
+\left \vert \mathbf{a}^{2}\right \vert +\left \vert \mathbf{b}^{2}\right \vert
\right) \left( \left \vert \mathbf{n}^{2}\right \vert +\left \vert \mathbf{a}%
\right \vert \right) -\left( n_{1}+2\left \vert \mathbf{n}^{2}\right \vert
+\left \vert \mathbf{a}\right \vert +\left \vert \mathbf{b}\right \vert
-1\right) \left( \left \vert \mathbf{n}^{2}\right \vert +\left \vert \mathbf{a}%
^{2}\right \vert +\frac{x_{1}}{2}\right) }{\left( n_{1}\right) _{2}} \\
&&+\frac{n_{1}+3\left \vert \mathbf{n}^{2}\right \vert +a_{1}+2\left \vert
\mathbf{a}^{2}\right \vert -b_{1}+\left \vert \mathbf{b}^{2}\right \vert +2-%
\frac{x_{1}}{2}}{n_{1}} \\
B_{3} &=&-\frac{\left( n_{1}+2\left \vert \mathbf{n}^{2}\right \vert
+\left \vert \mathbf{a}^{2}\right \vert +\left \vert \mathbf{b}^{2}\right \vert
+1\right) \left( \left \vert \mathbf{n}\right \vert +\left \vert \mathbf{a}%
\right \vert +1\right) }{\left( n_{1}\right) _{2}}.
\end{eqnarray*}
\end{theorem}

\begin{proof}
In view of the relation \eqref{P1}, if we apply the relation \eqref{101} consecutively by taking $a_{r}\rightarrow
n_{r}+a_{r}+a_{r+1}$ and $b_{r}\rightarrow n_{r}+b_{r}+b_{r+1}$, we obtain
the above relation.
\end{proof}

\begin{theorem}
For $r\geq 1$, the family of the special function$\ _{r}S_{\mathbf{n}}\left(
\mathbf{x};\mathbf{a},\mathbf{b}\right) $ verifies the following recurrence
relation%
\begin{multline*}
\left( B_{1}+C_{1}\right) \ _{r}S_{\mathbf{n}}\left( \mathbf{x};\mathbf{a},%
\mathbf{b}\right) +\left( B_{2}+C_{2}\right) \ _{r}S_{\mathbf{n}}\left(
\mathbf{x};\mathbf{a},b_{1}-1,b_{2},\dots ,b_{r+1}\right)  \\
+B_{3}\ _{r}S_{\mathbf{n}}\left( \mathbf{x};\mathbf{a},b_{1}-2,b_{2},\dots ,b_{r+1}%
\right) =0,
\end{multline*}%
where%
\begin{align*}
B_{1}+C_{1} &=-\frac{b_{1}-1+\frac{x_{1}}{2}}{n_{1}+2\left \vert \mathbf{n}%
^{2}\right \vert +\left \vert \mathbf{a}\right \vert +\left \vert \mathbf{b}%
\right \vert -1}, \\
B_{2}+C_{2} &=\frac{\left( 2\left \vert \mathbf{n}^{2}\right \vert
+\left \vert \mathbf{a}^{2}\right \vert +\left \vert \mathbf{b}^{2}\right \vert
\right) \left( \left \vert \mathbf{n}^{2}\right \vert +\left \vert \mathbf{a}%
\right \vert \right) +n_{1}\left( \left \vert \mathbf{n}^{2}\right \vert
+\left \vert \mathbf{a}^{2}\right \vert +\frac{x_{1}}{2}\right) }{\left(
n_{1}+2\left \vert \mathbf{n}^{2}\right \vert +\left \vert \mathbf{a}%
\right \vert +\left \vert \mathbf{b}\right \vert -2\right) _{2}} \\
&+\frac{n_{1}-\left \vert \mathbf{n}^{2}\right \vert -\left \vert \mathbf{a}%
^{2}\right \vert +2b_{1}-3+\frac{x_{1}}{2}}{n_{1}+2\left \vert \mathbf{n}%
^{2}\right \vert +\left \vert \mathbf{a}\right \vert +\left \vert \mathbf{b}%
\right \vert -1}, \\
B_{3} &=-\frac{\left( n_{1}+a_{1}+b_{1}-2\right) \left( \left \vert \mathbf{n%
}\right \vert +\left \vert \mathbf{b}\right \vert -2\right) }{\left(
n_{1}+2\left \vert \mathbf{n}^{2}\right \vert +\left \vert \mathbf{a}%
\right \vert +\left \vert \mathbf{b}\right \vert -2\right) _{2}}.
\end{align*}
\end{theorem}

\begin{proof}
If we apply the recurrence relation \eqref{102} successively by getting $n_{j}\rightarrow n_{j+1}$, $%
x_{j}\rightarrow x_{j+1}$, $a_{j}\rightarrow a_{j+1}$ and $b_{j}\rightarrow
b_{j+1}$ for $j=1,2,\dots ,r-1$, we obtain the required result in view of the relation \eqref{P2}.
\end{proof}

\begin{theorem}
For the family of the special function$\ _{r}S_{\mathbf{n}}\left(
\mathbf{x};\mathbf{a},\mathbf{b}\right) $, we get
\begin{multline*}
\left( B_{2}+C_{2}\right) \ _{r}S_{\mathbf{n}}\left(\mathbf{x}
;a_{1},\dots ,a_{r-1},a_{r}+1,a_{r+1}-1,b_{1},\dots ,b_{r-1},b_{r}-1,b_{r+1}+1\right)  \\
+B_{3}\ _{r}S_{\mathbf{n}}\left( \mathbf{x};a_{1},\dots ,a_{r-1},a_{r}+2,a_{r+1}-2,b_{1},\dots ,b_{r-1},b_{r}-2,b_{r+1}+2%
\right) \\
+\left( B_{1}+C_{1}\right) \ _{r}S_{\mathbf{n}}\left( \mathbf{x};\mathbf{a},\mathbf{b}\right) =0,
\end{multline*}
for $r\geq 1$ where
\begin{align*}
B_{1}+C_{1} &=-\frac{b_{r}-1+\frac{x_{r}}{2}}{a_{r+1}+\frac{x_{r}}{2}}, \\
B_{2}+C_{2} &=\frac{\left \vert \mathbf{a}^{r}\right \vert \left(
a_{r+1}+b_{r+1}\right) +n_{r}\left( n_{r}+\left \vert \mathbf{a}%
^{r}\right \vert +\left \vert \mathbf{b}^{r}\right \vert -1\right) }{\left(
a_{r+1}+\frac{x_{r}}{2}-1\right) _{2}} \\
&-\frac{\left \vert \mathbf{a}^{r}\right \vert -b_{r}+b_{r+1}+2-x_{r}}{%
a_{r+1}+\frac{x_{r}}{2}}, \\
B_{3} &=-\frac{\left( a_{r}-\frac{x_{r}}{2}+1\right) \left( b_{r+1}-\frac{%
x_{r}}{2}+1\right) }{\left( a_{r+1}+\frac{x_{r}}{2}-1\right) _{2}}.
\end{align*}
\end{theorem}

\begin{proof}
It is enough to consider the relation \eqref{P2} and use the recurrence relation \eqref{103} successively.
\end{proof}
Similarly, when the relations \eqref{104} and \eqref{106}, respectively, are applied consecutively in \eqref{P1} and \eqref{P2}, respectively, we can give next two theorems.
\begin{theorem}
For the family of the special function$\ _{r}S_{\mathbf{n}}\left(
\mathbf{x};\mathbf{a},\mathbf{b}\right) $,
\begin{multline*}
\left( B_{1}+C_{1}\right) \ _{r}S_{\mathbf{n}}\left( \mathbf{x};\mathbf{a},\mathbf{b}\right) +\left( B_{2}+C_{2}\right) \ _{r}S_{\mathbf{n}}\left(
\mathbf{x};a_{1}+1,a_{2},\dots ,a_{r+1},b_{1}-1,b_{2},\dots ,b_{r+1}\right)  \\
+C_{3}\ _{r}S_{\mathbf{n}}\left( \mathbf{x};a_{1}+2,a_{2},\dots ,a_{r+1},b_{1}-2,b_{2},\dots ,b_{r+1}\right) =0,
\end{multline*}%
holds for $r\geq 1$ where
\begin{align*}
B_{1}+C_{1} &=-\frac{b_{1}-1+\frac{x_{1}}{2}}{\left \vert \mathbf{n}%
^{2}\right \vert +\left \vert \mathbf{a}\right \vert -1}, \\
B_{2}+C_{2} &=\frac{\left( 2\left \vert \mathbf{n}^{2}\right \vert
+2\left \vert \mathbf{a}\right \vert +1\right) \left( 2\left \vert \mathbf{n}%
^{2}\right \vert +\left \vert \mathbf{a}^{2}\right \vert +\left \vert \mathbf{b}%
\right \vert -2+\frac{x_{1}}{2}\right) +n_{1}\left( n_{1}+2\left \vert \mathbf{%
n}^{2}\right \vert +\left \vert \mathbf{a}\right \vert +\left \vert \mathbf{b}%
\right \vert -1\right) }{{\left( \left \vert \mathbf{n}^{2}\right \vert +\left \vert
\mathbf{a}\right \vert -1\right) _{2}}} \\
&-\frac{\left( 2\left \vert \mathbf{n}^{2}\right \vert +\left \vert \mathbf{a}%
\right \vert +\left \vert \mathbf{b}\right \vert -1\right) \left( \left \vert
\mathbf{n}^{2}\right \vert +\left \vert \mathbf{a}^{2}\right \vert +\frac{x_{1}%
}{2}\right) }{\left( \left \vert \mathbf{n}^{2}\right \vert +\left \vert
\mathbf{a}\right \vert -1\right) _{2}}-\frac{2\left \vert \mathbf{n}%
^{2}\right \vert +\left \vert \mathbf{a}^{2}\right \vert +\left \vert \mathbf{b}%
^{2}\right \vert -1}{\left \vert \mathbf{n}^{2}\right \vert +\left \vert \mathbf{%
a}\right \vert -1}, \\
C_{3} &=-\frac{\left( \left \vert \mathbf{n}\right \vert +\left \vert \mathbf{a%
}\right \vert +1\right) \left( \left \vert \mathbf{n}\right \vert +\left \vert
\mathbf{b}\right \vert -2\right) \left( a_{1}-\frac{x_{1}}{2}+1\right) }{%
\left( \left \vert \mathbf{n}^{2}\right \vert +\left \vert \mathbf{a}%
\right \vert -1\right) _{3}}.
\end{align*}
\end{theorem}

\begin{theorem}
For the family of the special function$\ _{r}S_{\mathbf{n}}\left(
\mathbf{x};\mathbf{a},\mathbf{b}\right) $, we have
\begin{multline*}
\left( B_{1}+C_{1}\right) \ _{r}S_{\mathbf{n}}\left( \mathbf{x};\mathbf{a},%
\mathbf{b}\right) +\left( B_{2}+C_{2}\right) \ _{r}S_{\mathbf{n}}\left(
\mathbf{x};\mathbf{a},b_{1},\dots ,b_{r-1},b_{r}-1,b_{r+1}+1\right)  \\
+C_{3}\ _{r}S_{\mathbf{n}}\left( \mathbf{x};\mathbf{a}%
,b_{1},\dots ,b_{r-1},b_{r}-2,b_{r+1}+2\right) =0,
\end{multline*}%
for  $r\geq 1$ where
\begin{align*}
B_{1}+C_{1} &=-\frac{b_{r}-1+\frac{x_{r}}{2}}{a_{r+1}+b_{r+1}-1}, \\
B_{2}+C_{2} &=\frac{n_{r}\left( n_{r}+\left \vert \mathbf{a}^{r}\right \vert
+\left \vert \mathbf{b}^{r}\right \vert -1\right) -\left( a_{r+1}+\frac{x_{r}}{%
2}\right) \left( \left \vert \mathbf{a}^{r}\right \vert +\left \vert \mathbf{b}%
^{r}\right \vert -1\right) }{\left( a_{r+1}+b_{r+1}-1\right) _{2}} \\
&+\frac{\left( 2a_{r+1}+2b_{r+1}+1\right) \left( \left \vert \mathbf{a}%
^{r}\right \vert +b_{r}-2+\frac{x_{r}}{2}\right) }{\left(
a_{r+1}+b_{r+1}-1\right) _{2}}-\frac{\left \vert \mathbf{a}^{r}\right \vert -1%
}{a_{r+1}+b_{r+1}-1}, \\
C_{3} &=-\frac{\left( n_{r}+a_{r}+b_{r}-2\right) \left(
n_{r}+a_{r+1}+b_{r+1}+1\right) \left( b_{r+1}-\frac{x_{r}}{2}+1\right) }{%
\left( a_{r+1}+b_{r+1}-1\right) _{3}}.
\end{align*}
\end{theorem}

\begin{theorem}
The family of the special function$\ _{r}S_{\mathbf{n}}\left( \mathbf{x};%
\mathbf{a},\mathbf{b}\right) $ satisfies the following recurrence relation%
\begin{multline*}
\left( \left \vert \mathbf{n}^{2}\right \vert +\left \vert \mathbf{a}%
\right \vert \right) \ _{r}S_{\mathbf{n}}\left( \mathbf{x};\mathbf{a},\mathbf{%
b}\right) \\
-\left( n_{1}+2\left \vert \mathbf{n}^{2}\right \vert +\left \vert
\mathbf{a}\right \vert +\left \vert \mathbf{b}\right \vert -1\right) \ _{r}S_{%
\mathbf{n}}\left( \mathbf{x};a_{1}+1,a_{2},\dots ,a_{r+1},b_{1},\dots ,b_{r+1}%
\right)  \\
+\left( \left \vert \mathbf{n}\right \vert +\left \vert \mathbf{b}\right \vert
-1\right) \ _{r}S_{\mathbf{n}}\left( \mathbf{x}%
;a_{1}+1,a_{2},\dots ,a_{r+1},b_{1}-1,b_{2},\dots ,b_{r+1}\right) =0
\end{multline*}%
for $r\geq 1$.
\end{theorem}

\begin{proof}
By using the relation \eqref{107} consecutively by taking into account \eqref{P2}, the relation above is found.
\end{proof}

\section*{Acknowledgements}

The work of I.A. has been partially supported by the Agencia Estatal de Investigaci\'on (AEI) of Spain under Grant MTM2016-75140-P, cofinanced by the European Community fund FEDER.


\end{document}